%% file: main-text.tex
\documentclass{article}

\usepackage[twoside,
paperwidth=210mm,
paperheight=297mm,
textheight=622pt,
textwidth=468pt,
centering,
headheight=50pt,
headsep=12pt,
footskip=18pt,
footnotesep=24pt plus 2pt minus 12pt,
columnsep=2pc]{geometry}

\usepackage[auth-sc]{authblk}

\input{latex-commands.tex}

\usepackage{amsthm}

\usepackage[capitalise]{cleveref}

\newtheorem{theorem}{Theorem}[section]
\newtheorem{lemma}[theorem]{Lemma}
\newtheorem{remark}[theorem]{Remark}

\theoremstyle{definition}

\newtheorem{assumption}{Assumption}

\theoremstyle{remark}

\crefname{equation}{}{}

\usepackage{lineno}
\usepackage[numbers,sort&compress]{natbib}

\numberwithin{figure}{section}
\numberwithin{table}{section}
\numberwithin{algorithm}{section}

\title{Multiscale Modeling for Time-harmonic Maxwell equations with impedance boundary conditions in highly heterogeneous media}
\author[1]{Xiang Zhong}
\author[1]{Eric T. Chung}
\author[1]{Xingguang Jin\thanks{Corresponding author.
(Email address: \href{mailto:xgjin@math.cuhk.edu.hk}{xgjin@math.cuhk.edu.hk})}}
\affil[1]{Department of Mathematics, The Chinese University of Hong Kong, Shatin, Hong~Kong~SAR, China.}
\date{}
\begin{document}
\maketitle

\begin{abstract}
  \input{abstract.tex}
  
\end{abstract}
\section{Introduction}
The study of time-harmonic Maxwell's equations is fundamental in modern electromagnetics, especially with the rapid advancement of electromagnetic metamaterials. These are artificially engineered structures that manipulate waves in ways not possible with natural materials \cite{Veselago1967,Pendry2000}. By arranging subwavelength meta-atoms, metamaterials can achieve unusual properties such as negative refraction, cloaking, and superlensing \cite{Lipton2018,Smith2004,Leonhardt2006,Lamacz2016}. Photonic crystals are key examples that utilize periodic dielectric variations to engineer photonic band structures via Bragg scattering. Such lattices can either suppress wave propagation through photonic band gaps or facilitate unique phenomena like slow light and self-collimation \cite{Meade2008,Sakoda2005}.

Despite substantial mathematical progress in understanding complex electromagnetic phenomena \cite{Monk2003}, the simulation of wave propagation in highly heterogeneous media remains a formidable challenge. Traditional numerical schemes for the time-harmonic Maxwell’s equations, such as standard finite element or finite difference methods, often become computationally prohibitive in the presence of microscale heterogeneities. Resolving these fine-scale features necessitates extremely refined meshes, leading to excessive computational costs and ill-conditioned algebraic systems \cite{Hiptmair2002}. Beyond the multiscale nature of the problem, the high-wavenumber regime remains particularly challenging despite various numerical demonstrations. Moreover, the emergence of unconventional coefficients and the non-coercive nature inherent in metamaterial modeling frequently undermine traditional numerical analysis frameworks. While classical homogenization theory offers a path toward effective macroscopic models, its applicability is often limited by restrictive structural assumptions \cite{Bensoussan2011}.

To address these difficulties, multiscale methods have emerged as a powerful paradigm, designed to embed fine-scale information directly into coarse-scale basis functions. A diverse array of such techniques has been developed, including multiscale finite element methods (MsFEMs) \cite{HTY1997, YTY2009}, {general multiscale finite element methods (GMsFEM) \cite{Chung2019}}, localized orthogonal decomposition (LOD) \cite{AMDP2014}, variational multiscale methods \cite{HFJL1998}, numerical upscaling \cite{PDR2016}, heterogeneous multiscale methods (HMM) \cite{EW2003, EWBE2003, BEYH2005}, and {numerical homogenization \cite{Wang2026}}. Specifically, in the context of electromagnetic waves, significant advancements have been made through multiscale asymptotic methods and HMMs to construct efficient coarse-scale models~\cite{Cao2010, Ciarlet2017, Henning2016}, with extensions addressing time-domain analysis~\cite{Hochbruck2017}, high-contrast materials~\cite{Verfuerth2019}, and LOD-based computational homogenization~\cite{Henning2020}.
Despite their success, the performance of these approaches generally hinges on the explicit construction of stable correctors or specific structural properties such as periodicity. 
In contrast, the Constrained Energy Minimization Generalized Multiscale Finite Element Method (CEM-GMsFEM) \cite{chung2018, chung2023multiscale,JinLiuZhongChung2025, Zhou2026,Chung2026} provides a robust alternative by constructing multiscale basis functions via energy minimization subject to local constraints on the coarse grid. This framework naturally accommodates high-contrast coefficients without the need for specific quasi-interpolation stability, allows for systematic control over the energy-norm error, and offers the flexibility to incorporate local spectral information. Consequently, CEM-GMsFEM delivers a highly efficient computational methodology for challenging electromagnetic simulations involving complex geometries and extreme heterogeneities.

{In this work, we consider} the time-harmonic Maxwell equations at a fixed angular frequency $\omega > 0$ in free space (i.e. the relative permittivity $\varepsilon_r \equiv 1$) containing highly heterogeneous magnetic material described by the relative permeability $\mu_r(\mathbf{x})\in L^\infty(\Omega)$ that may vary over several orders of magnitude. More precisely, for $\mu_r(\mathbf{x})$, we suppose there exist positive constants $0 <\mu_{\mathup{min}} \leq \mu_{\mathup{max}}$ such that for a.e. $\mathbf{x}\in\Omega$, $\mu_{\mathup{min}}\leq\mu_r(\mathbf{x})\leq\mu_{\mathup{max}}$. We assume the system is non-conductive (i.e. the conductivity $\sigma \equiv 0$). Let $\Omega \subset \mathbb{R}^3$ be a bounded connected Lipschitz domain.
The problem is to find the electric field $\bm{u} \in \mathbb{C}^3$ corresponding to a given current density $\mathbf{f}$ such that the following system holds:
\begin{subequations}\label{model}
\begin{alignat}{2}
\operatorname{curl}\left(\mu_r^{-1} \operatorname{curl} \mathbf{u}\right) - k^2 \mathbf{u} &=  \mathbf{f}  &\quad& \text{in } \Omega, \label{model_a} \\
\mu_r^{-1} \operatorname{curl} \mathbf{u} \times \mathbf{n} - i k \mathbf{u}_T &= \mathbf{g} &\quad& \text{on } \partial\Omega, \label{model_b}
\end{alignat}
\end{subequations}
where $k = \omega/c_0 =\omega\sqrt{\varepsilon_0 \mu_0} > 0$ is the (fixed) free-space wavenumber. Here $\varepsilon_0$ and $\mu_0$ denote the vacuum permittivity and permeability, respectively, and $c_0 = 1/\sqrt{\varepsilon_0 \mu_0}$ the speed of light in vacuum. $\mathbf{g}$ is a given tangential vector field on the boundary $\partial\Omega$. $\mathbf{n}$ is the unit outward normal vector on $\partial\Omega$. $\mathbf{w}_T = (\mathbf{n} \times \mathbf{w}) \times \mathbf{n} = \mathbf{w} - (\mathbf{w} \cdot \mathbf{n}) \mathbf{n}$ denotes the tangential component trace. The boundary condition (\ref{model_b}) is the first-order Silver--M\"uller absorbing condition, which is just an impedance boundary condition with the impedance parameter $\lambda = 1$ (also with $\mu_r=1$ on $\partial\Omega$ and $\varepsilon_r=\mu_r=1$ in a neighborhood of $\partial\Omega$, see \cite{Monk2003}). The presence of the imaginary unit $i$ introduces a phase shift, which allows the boundary to partially transmit and reflect incoming waves. This key feature enables the simulation of wave radiation into open space, making it possible to model unbounded wave scattering and radiation problems within a finite computational domain $\Omega$. {We restrict ourselves to the case \(\varepsilon_r\equiv 1\) primarily to isolate the effect of high-contrast permeability. More precisely, we focus on heterogeneous media with high-contrast permeability \(\mu_r\) at a fixed wavenumber $k$, rather than on the additional complications caused by simultaneous high contrast in both constitutive parameters or by the high-frequency asymptotic regime. Similar one-parameter settings are commonly used to highlight the dominant physical mechanism, such as dielectric-dominant or magnetic-dominant responses \cite{Pendry1999,Holloway2009}.}
Physically, this model describes electromagnetic wave scattering and propagation in air or vacuum, perturbed by the presence of strongly heterogeneous, purely magnetic inclusions or metamaterials. Typical applications include the modeling of ferrites, magnetic cloaks, $\mu$-near-zero structures, magnetic photonic crystals, and magnetic radar-absorbing materials.

We propose an efficient multiscale framework for time-harmonic Maxwell equations with impedance boundary conditions in heterogeneous media. The method constructs multiscale basis functions tailored for highly heterogeneous media in two stages.
In the first stage, an auxiliary multiscale space is built for each coarse element by solving local spectral problems. Notably, no divergence-free constraint—either weak or strong—is imposed during this construction. This is because the local spectral problem incorporates a mass term and a Silver–Müller-type boundary penalty term, which together ensure the coercivity of the associated bilinear form. As a result, the kernel of the curl operator is automatically excluded from the dominant eigenspaces, eliminating the need for explicit enforcement of the divergence-free condition on the local auxiliary multiscale basis functions. The design of the local spectral problems proposed in this study represents a novel departure from existing approaches, a feature that has not been fully addressed in prior literature \cite{Cao2010, Ciarlet2017, Henning2016, Hochbruck2017, Verfuerth2019, Henning2020}.
In the second stage, the auxiliary multiscale space is used to construct the final multiscale space. Unlike traditional construction strategies \cite{chung2018,chung2025locking}, the bilinear form employed in this stage differs from the coercive form used in the first stage. Therefore, establishing the coercivity of a more complex bilinear form becomes essential for the subsequent theoretical analysis. Leveraging a suitable resolution condition and relationships between various norms, we establish the desired coercivity.
By appropriately selecting the number of graph oversampling layers, we achieve an $O(H)$ convergence that is independent of the local contrast. Meanwhile, the approximation error increases with the wave number. 

{We provide more detailed comparisons with several existing multiscale methods for Maxwell equations \cite{Ciarlet2017,Henning2016,Henning2020,Chung2019,Wang2026}. HMM-based methods \cite{Ciarlet2017,Henning2016} mainly focus on periodic or scale-separated media, while the LOD framework \cite{Henning2020} relies on the construction of stable correctors and may require restrictive resolution conditions in strongly heterogeneous settings.} {The adaptive GMsFEM of Chung and Li \cite{Chung2019} was developed for coercive $H(\mathrm{curl})$-elliptic problems and therefore differs fundamentally from the indefinite Maxwell setting considered here. More recently, \cite{Wang2026} proposed a numerical homogenization approach for indefinite time-harmonic Maxwell equations based on an edge multiscale approach. For these Maxwell multiscale methods, the treatment of the kernel of the curl operator is often a central issue and may involve divergence-free constraints, auxiliary saddle-point formulations, or additional projection procedures. In contrast, the proposed local spectral problem in our work incorporates a positive mass term and a Silver--Müller-type boundary penalty, which automatically suppresses curl-kernel modes in the auxiliary space construction. From an implementation viewpoint, the proposed spectral problem avoids imposing explicit divergence-free constraints in both the local spectral problems and the multiscale basis construction, thereby simplifying the local basis generation procedure.}
{Moreover, compared with HMM \cite{Ciarlet2017,Henning2016} and LOD \cite{Henning2020} approaches,
our work is particularly attractive for high-contrast heterogeneous media. This allows the method to accommodate more complicated heterogeneous structures encountered in practical applications. The auxiliary space is obtained through local eigenvalue problems that directly identify the dominant multiscale features associated with the coefficient heterogeneity. This construction leads to a contrast-independent \(O(H)\) convergence theory and provides a systematic mechanism for selecting multiscale basis functions that remain robust in high-contrast media.}

{We emphasize that the present work goes beyond a direct extension of existing CEM-GMsFEM techniques \cite{chung2018, chung2023multiscale,JinLiuZhongChung2025}.The main challenge is to construct a Maxwell-compatible multiscale framework capable of handling the non-coercive and non-Hermitian structure of the time-harmonic Maxwell system. In particular, we introduce a new local spectral problem that avoids explicit divergence-free constraints, develop a Petrov--Galerkin multiscale formulation with distinct trial and test spaces, and establish new stability and localization property that are required for the indefinite and non-Hermitian Maxwell setting -- results that do not follow directly from the existing CEM-GMsFEM theory. These developments lead to a rigorous, contrast-independent $O(H)$ convergence for heterogeneous Maxwell problems with impedance boundary conditions.}

{We also point out that the proposed multiscale framework is not limited to $\varepsilon_r = 1$. Our analysis can be adapted to the case where $\varepsilon_r(\mathbf{x})$ is also heterogeneous. Indeed, the extension mainly requires replacing the standard mass term by the weighted term involving $\varepsilon_r$, while the multiscale basis construction, localization analysis, and stability arguments can be adapted with minor modifications. If 
\(\varepsilon_r(\mathbf{x})\) also exhibits high contrast, the corresponding resolution condition may become more restrictive due to the additional dependence on the contrast of \(\varepsilon_r\).}

This paper is organized as follows. In \cref{sec: Preliminaries}, we introduce some notation and definitions. The construction of the multiscale basis functions in the proposed method is described in \cref{sec:multiscale method}. All theoretical results and error analyses for the proposed method are presented in \cref{sec:ana}. To assess the performance of the proposed approach, numerical experiments on three representative models are reported in \cref{sec:Numerical experiments}. Finally, conclusions and perspectives for future work are given in \cref{sec:conclusions}.
\section{Preliminaries}\label{sec: Preliminaries}
In this paper, vector-valued functions are denoted by boldface letters and, unless specified, are complex-valued. Throughout this paper, we use standard notation: For a domain \(D\in\Omega\), \(L^2(D)\) denotes the usual complex Lebesgue space with norm \(\|\cdot\|_{L^p(D)}\). If $D = \Omega$, then we may drop the subscript $D$. Vector space is represented in bold black italics, for instance, $\bm{L}^2(D)\coloneqq[L^2(D)]^3$.
The dot denotes the standard scalar product. We adopt the convention that the complex scalar product is conjugate-linear in the second argument, with \(\overline{\mathbf{v}}\) denoting the complex conjugate of \(\mathbf{v}\). The vector-valued Hilbert space $\bm{H}(\mathrm{curl}, D)$ is defined as $\bm{H}(\mathrm{curl}, D)\coloneqq\{\mathbf{u} \in \bm{L}^2(D) \colon \mathrm{curl} \, \mathbf{u} \in \bm{L}^2(D) \}$, equipped with its standard graph norm scalar product $(\cdot, \cdot)_{H(\mathrm{curl}, D)}$.

Define the following vector spaces
$$
\begin{aligned}
\bm{L}_t^2(\partial \Omega) &\coloneqq \left\{ \mathbf{v} \in \bm{L}^2(\partial \Omega) \colon \mathbf{v} \cdot \mathbf{n} = 0 \right\}, \\
\bm{V} &\coloneqq \bm{H}_{\mathrm{imp}}(\operatorname{curl}; \Omega) = \left\{ \mathbf{v} \in \bm{H}(\operatorname{curl}, \Omega) \colon \mathbf{v}_T \in \bm{L}^2(\partial \Omega) \right\}
\end{aligned}
$$
and the norms: for any $D\subset\Omega$,
$$
\begin{aligned}
\norm{\mathbf{v}}_{L^2(D)}^2& =\int_D\mathbf{v}\cdot\overline{\mathbf{v}}dx,\quad \norm{\mathbf{v}}_{L^2(\partial D)}^2=\int_{\partial D}\mathbf{v}\cdot\overline{\mathbf{v}}ds,\\
\| \mathbf{v} \|_{k, \mathrm{imp}, D}^2 &= \| \operatorname{curl} \mathbf{v} \|_{L^2(D)}^2 + k^2 \| \mathbf{v} \|_{L^2(D)}^2 + k \| \mathbf{v}_T \|_{L^2(\partial D)}^2.
\end{aligned}
$$

We assume the current density $\mathbf{f}\in \bm{L}^2(\Omega)$ and the tangential vector $\mathbf{g}\in \bm{L}_t^2(\partial \Omega)$. The weak form of problem (\ref{model}) is to find $\mathbf{u}\in \bm{V}$ such that 
$$
\begin{aligned}
& \int_{\Omega} \left( \mu_r^{-1} \operatorname{curl} \mathbf{u} \right) \cdot \operatorname{curl} \overline{\mathbf{v}} \, dx - k^2\int_{\Omega} \mathbf{u}\cdot\overline{\mathbf{v}} \, dx - i k \int_{\partial \Omega} \mathbf{u}_T \cdot\overline{\mathbf{v}}_T \, ds = \int_{\Omega} \mathbf{f} \cdot\overline{\mathbf{v}} \, dx + \int_{\partial \Omega} \mathbf{g}\cdot \overline{\mathbf{v}}_T \, ds.
\end{aligned}
$$
for all $\mathbf{v}\in\bm{V}$.
In order to simplify notation, we utilize the following inner products. For any $\mathbf{w},\mathbf{v}\in\bm{L}^2(D)$ (where $D\subset\Omega$) (note that the subscript $D$ may be dropped when $D=\Omega$)
\[
(\mathbf{w},\mathbf{v})_D=\int_D\mathbf{w}\cdot\overline{\mathbf{v}}dx,\quad \left\langle \mathbf{w}_T, \mathbf{v}_T \right\rangle_{\partial D} = \int_{\partial D} \mathbf{w}_T\cdot\mathbf{v}_T \, ds.
\]
Furthermore, we define the sesquilinear for $B:\mathbf{V}\times\mathbf{V}\to\mathbb{C}$ as follows
$$
B(\mathbf{w}, \mathbf{v}) = \left( \mu_r^{-1} \operatorname{curl} \mathbf{w}, \operatorname{curl} \mathbf{v} \right) - k^2\left( \mathbf{w}, \mathbf{v} \right) - i k \left\langle \mathbf{w}_T, \mathbf{v}_T \right\rangle.
$$
Using this notation, the variational problem is to find $\mathbf{u}\in\bm{V}$ such that
\begin{equation}
\label{weak form}
B(\mathbf{u}, \mathbf{v}) = (\mathbf{f}, \mathbf{v}) + \left\langle \mathbf{g}, \mathbf{v}_T \right\rangle,\quad \forall \mathbf{v}\in\bm{V}.
\end{equation}
Clearly we have the following estimate:
\begin{equation}
\label{boundedness}
| B(\mathbf{w}, \mathbf{v}) | \leq \max \left\{\mu_{\min}^{-1}, 1 \right\} \| \mathbf{w} \|_{k, \mathrm{imp}} \cdot \| \mathbf{v} \|_{k, \mathrm{imp}} \quad \forall \mathbf{w}, \mathbf{v} \in \bm{V}.
\end{equation}

In terms of \cite[Theorem 4.17]{Monk2003}, we know problem (\ref{weak form}) possesses a unique solution $\mathbf{u}\in\bm{V}$ for any value of $k>0$. Furthermore, there is a constant $C_k>0$ independent of $\mathbf{u}, \mathbf{f}$ and $\mathbf{g}$ but depending on $k$ such that
\begin{equation}
\label{prior estimate}
\|\mathbf{u}\|_{k, \mathrm{imp}} \leq C_k\left(\|\mathbf{f}\|_{L^2}+\|\mathbf{g}\|_{\left(L^2(\partial \Omega)\right)}\right).
\end{equation}
Then, using the analysis similar to \cite[Lemma 2.1]{Peterseim2017}, the estimate (\ref{prior estimate}) for any $(\mathbf{f},\mathbf{g})\in\bm{L}^2(\Omega)\times\bm{L}_t^2(\partial \Omega)$ implies well-posedness, i.e.
the inf-sup condition as follows
\begin{equation}
\label{continuous inf-sup}
\inf _{\mathbf{w} \in \bm{V}\ \{\mathbf{0}\}} \sup _{\mathbf{v} \in \bm{V}\ \{\mathbf{0}\}} \frac{|B(\mathbf{w}, \mathbf{v})|}{\|\mathbf{w}\|_{k, \mathup{imp}} \cdot\|\mathbf{v}\|_{k, \mathup{imp}}} \geqslant \frac{1}{ 2C_k\max\{\sqrt{k},k\}}>0 .
\end{equation}
Let $\mathcal{T}_H\coloneqq\cup_{i=1}^N{K_i}$ denote a conforming quasi-uniform partition of the three-dimensional domain $\Omega$ into hexahedral (cube) elements, where $H$ represents the coarse mesh size and $N$ is the total number of coarse elements. We refer to $\mathcal{T}_H$ as the coarse mesh, where each coarse element $K_i$ is further subdivided into a connected union of smaller fine-grid cubes, and { $K_{i,m} \subset \Omega$ is the oversampling coarse region by enlarging $K_i$ by $m$ coarse grid layers.}  The corresponding fine mesh, denoted by $\mathcal{T}_h\coloneqq\cup_{i=1}^{N_h}{T_i}$ (with $N_h$ being the number of fine cubic elements), is constructed as a uniform refinement of $\mathcal{T}_H$.
For an illustrative example, see Figure \ref{fig:3d-dashed-lines}, which provides a three-dimensional visualization of the coarse cubic mesh, the fine cubic mesh,  and { an oversampling region extending one coarse layer outward from a selected coarse element $K_i$ when $m=1$. }

\begin{figure}[!ht]
\centering
\begin{tikzpicture}[
    scale=0.8,
    x={(-0.9cm,-0.5cm)}, y={(0.9cm,-0.5cm)}, z={(0cm,1cm)},
    fine/.style={gray!30, very thin},
    coarse/.style={black, thick},
    coarsehidden/.style={black!60, thick, dashed, dash pattern=on 4pt off 3pt},
    deepblue/.style={fill=blue!70!black, draw=blue!90!black, very thick, opacity=0.90},
    lightblue/.style={fill=blue!25, draw=blue!50, very thick, opacity=0.35},
    redtau/.style={fill=red!75, draw=red!90, thick},
    hidden/.style={dashed, dash pattern=on 4pt off 3pt, opacity=0.7}
  ]

\def\H{1.0}  
\def\fineScale{0.25}  
\def\domainSize{4}    


\foreach \j in {0,...,16} {
  \foreach \k in {0,...,16} {
    \draw[fine]
      (0,\j*\fineScale,\k*\fineScale)
      -- (16*\fineScale,\j*\fineScale,\k*\fineScale);
  }
}

\foreach \i in {0,...,16} {
  \foreach \k in {0,...,16} {
    \draw[fine]
      (\i*\fineScale,0,\k*\fineScale)
      -- (\i*\fineScale,16*\fineScale,\k*\fineScale);
  }
}

\foreach \i in {0,...,16} {
  \foreach \j in {0,...,16} {
    \draw[fine]
      (\i*\fineScale,\j*\fineScale,0)
      -- (\i*\fineScale,\j*\fineScale,16*\fineScale);
  }
}

\tikzset{
  innergrid/.style={black!60, dashed, thin}
}
\foreach \z in {0,...,\domainSize} {
  \foreach \i in {1,...,\numexpr\domainSize-1} {
    \draw[innergrid] (\i*\H,0,\z*\H) -- (\i*\H,\domainSize*\H,\z*\H);
  }
  \foreach \j in {1,...,\numexpr\domainSize-1} {
    \draw[innergrid] (0,\j*\H,\z*\H) -- (\domainSize*\H,\j*\H,\z*\H);
  }
}


\foreach \i in {1,...,\numexpr\domainSize-1} {
  \foreach \j in {1,...,\numexpr\domainSize-1} {
    \draw[innergrid] (\i*\H,\j*\H,0) -- (\i*\H,\j*\H,\domainSize*\H);
  }
}

\draw[black, very thick]
  (0,0,0) --
  (\domainSize*\H,0,0) --
  (\domainSize*\H,\domainSize*\H,0) --
  (0,\domainSize*\H,0) -- cycle;
\draw[black, very thick, dashed]
  (0,0,\domainSize*\H) --
  (\domainSize*\H,0,\domainSize*\H) --
  (\domainSize*\H,\domainSize*\H,\domainSize*\H) --
  (0,\domainSize*\H,\domainSize*\H) -- cycle;
\draw[black, very thick]
  (0,0,0) -- (0,0,\domainSize*\H)
  (\domainSize*\H,0,0) -- (\domainSize*\H,0,\domainSize*\H)
  (\domainSize*\H,\domainSize*\H,0) -- (\domainSize*\H,\domainSize*\H,\domainSize*\H/4)
 (0,\domainSize*\H,0) -- (0,\domainSize*\H,\domainSize*\H);
\node[black, font=\bfseries] 
at (\domainSize*\H+1,\domainSize*\H+1 ,0) {$\Omega$};

\fill[lightblue,dashed] (0.5*\H,0.5*\H,0.5*\H) -- (3.5*\H,0.5*\H,0.5*\H) -- 
                 (3.5*\H,3.5*\H,0.5*\H) -- (0.5*\H,3.5*\H,0.5*\H) -- cycle;
\fill[lightblue,dashed] (0.5*\H,0.5*\H,3.5*\H) -- (3.5*\H,0.5*\H,3.5*\H) -- 
                 (3.5*\H,3.5*\H,3.5*\H) -- (0.5*\H,3.5*\H,3.5*\H) -- cycle;

\draw[blue!50, very thick] 
   (3.5*\H,0.5*\H,0.5*\H) -- 
  (3.5*\H,3.5*\H,0.5*\H) -- (0.5*\H,3.5*\H,0.5*\H) 
  (0.5*\H,3.5*\H,0.5*\H)--(0.5*\H,3.5*\H,3.5*\H) 
  (3.5*\H,3.5*\H,0.5*\H) -- (3.5*\H,3.5*\H,3.5*\H);
\draw[blue!50, very thick, hidden] 
  (0.5*\H,0.5*\H,3.5*\H) -- (3.5*\H,0.5*\H,3.5*\H) -- 
  (3.5*\H,3.5*\H,3.5*\H) -- (0.5*\H,3.5*\H,3.5*\H) -- cycle;
\draw[blue!50, very thick] 
  (0.5*\H,0.5*\H,0.5*\H) -- (0.5*\H,0.5*\H,3.5*\H)
  (3.5*\H,0.5*\H,0.5*\H) -- (3.5*\H,0.5*\H,3.5*\H);


\draw[blue!90!black, very thick,hidden]
  (1.5*\H,1.5*\H,1.5*\H) -- (2.5*\H,1.5*\H,1.5*\H) -- 
  (2.5*\H,2.5*\H,1.5*\H) -- (1.5*\H,2.5*\H,1.5*\H) -- cycle;
\draw[blue!90!black, very thick, hidden]
  (1.5*\H,1.5*\H,2.5*\H) -- (2.5*\H,1.5*\H,2.5*\H) -- 
  (2.5*\H,2.5*\H,2.5*\H) -- (1.5*\H,2.5*\H,2.5*\H) -- cycle;
\draw[blue!90!black, very thick,hidden] 
  (1.5*\H,1.5*\H,1.5*\H) -- (1.5*\H,1.5*\H,2.5*\H)
  (2.5*\H,1.5*\H,1.5*\H) -- (2.5*\H,1.5*\H,2.5*\H);
\draw[blue!90!black, very thick, hidden] 
  (1.5*\H,2.5*\H,1.5*\H) -- (1.5*\H,2.5*\H,2.5*\H)
  (2.5*\H,2.5*\H,1.5*\H) -- (2.5*\H,2.5*\H,2.5*\H);

\node[blue!90!black, font=\bfseries, fill=white, inner sep=2pt] (k)
at  (7*\H, 5*\H, 2*\H) {$K_i$};
\draw[->, blue!90!black, thick,
      shorten <=4pt, shorten >=4pt]
(k.east) -- (1.5*\H,1.5*\H,1.5*\H);
\begin{scope}[shift={(1.8*\H,3*\H,1.8*\H)}, scale=0.25]
  \fill[redtau] (0,0,0) rectangle (1,1);

  \fill[redtau] (0,0,1) rectangle (1,1);

  \draw[red!90,  thick] (0,0,0) rectangle (1,1);

  \draw[red!90, thick] (0,0,1) rectangle (1,1);

  \draw[red!90, thick]
    (0,0,0) -- (0,1,0)
    (0,1,0) -- (1,1,0)
    (1,1,0) -- (1,0,0)
    (1,1,1) -- (0,1,1)
    (0,1,1) -- (0,0,1)
    (0,0,1) -- (1,0,1)
    (1,0,1) -- (1,1,1)
    (1,1,1) -- (1,0,1)
    (0,0,0) -- (0,0,1)
    (1,0,0) -- (1,0,1)
    (0,1,0) -- (0,1,1)
    (1,1,0) -- (1,1,1);

 \node[white, font=\bfseries\scriptsize] at (0.5,0.5,0.5) {$\tau$};

\end{scope}

\node[red!90!black, font=\bfseries, fill=white, inner sep=2pt] (taulabel)
at (-1.8*\H,3*\H,1.8*\H) {$\tau$};

\draw[->, red!90!black, thick,
      shorten <=4pt, shorten >=4pt]
([xshift=-8pt]taulabel.east) -- (1.9*\H,3.1*\H,1.9*\H);

\node[black, font=\bfseries] at (\domainSize*\H/2, -1, 0) 
      {};
\node[gray!70, font=\bfseries] at (-1, \domainSize*\H/2, 0) 
      {};
\node[blue!80!black, font=\bfseries, fill=white, inner sep=2pt] (Kione)
at (8*\H,2*\H,4*\H) {$K_{i,1}$};
\draw[->, blue!80!black, thick,
  shorten >=10pt]
(Kione.east) -- (3.5*\H,2*\H,2*\H);
\draw[->, very thick] (0,0,0) -- (\domainSize*\H+0.8,0,0)
      node[below right] {$x$};

\draw[->, very thick] (0,0,0) -- (0,\domainSize*\H+0.8,0)
      node[below left] {$y$};

\draw[->, very thick] (0,0,0) -- (0,0,\domainSize*\H+0.8)
      node[above] {$z$};
\end{tikzpicture}
\caption{Three-dimensional illustration of nested meshes $\mathcal{T}_h$ and $\mathcal{T}_H$.
A coarse element $K_i$ (dark blue) is shown with its corresponding oversampling region $K_{i,1}$ (light blue),
which extends by one coarse element layer in all directions. 
A fine element $\tau \in \mathcal{T}_h$ is highlighted in red inside $K_i$.}
\label{fig:3d-dashed-lines}
\end{figure}

\section{The multiscale method} \label{sec:multiscale method}
In this section, we will present the construction of our multiscale method. The construction of the basis functions are developed on the coarse mesh illustrated in Figure~\ref{fig:3d-dashed-lines} and divided into two stages.
The first stage consists of constructing the auxiliary multiscale space (Section \ref{Auxiliary}). 
In the second stage, we will use the auxiliary multiscale space to construct multiscale space (Section \ref{Multiscale basis}). 

\subsection{Auxiliary multiscale space}\label{Auxiliary}
We will construct a set of auxiliary multiscale basis functions for each coarse element $K_i$ by solving a local spectral problem. For a general set $R$, let $\bm{V}(R)$ be the restriction of $\bm{V}$ on $R$.
Then we define the required spectral problem. For each coarse element $K_i$, we solve the eigenvalue problem:
find eigenpairs $\left(\lambda^i_j, \bm{\phi}^i_j\right) \in \mathbb{R} \times \bm{V}(K_i)$ such that
\begin{equation}
\label{local spectral problem}
a_i\left(\bm{\phi}^i_j, \mathbf{v}\right)=\lambda^i_j s_i\left(\bm{\phi}^i_j, \mathbf{v}\right) \quad \forall \mathbf{v} \in \bm{V}(K_i),
\end{equation}
where
\begin{subequations}
\label{def of a and s}
\begin{align}		
a_i(\mathbf{w}, \mathbf{v})& =  \int_{K_i} \mu_r^{-1}(\operatorname{curl} \mathbf{w}) \cdot (\operatorname{curl} \overline{\mathbf{v}}) \, dx 
 + k^2 \int_{K_i} \mathbf{w} \cdot \overline{\mathbf{v}} \, dx + k \int_{\partial K_i \cap \partial \Omega} \mathbf{w}_T \cdot \overline{\mathbf{v}}_T \, ds,\\
s_i(\mathbf{w}, \mathbf{v})& = \int_{K_i}\mu_r^{-1}H^{-2}\mathbf{w}\cdot\overline{\mathbf{v}}dx,
\end{align}
\end{subequations}
for all $\mathbf{w}, \mathbf{v}\in \bm{V}(K_i)$.
\begin{remark}
\label{no need for divergence free}
{We emphasize that the local spectral problem (\ref{local spectral problem})-(\ref{def of a and s}) is not obtained by directly restricting the original complex-valued sesquilinear form \(B(\cdot,\cdot)\) to \(K_i\). Since \(B(\cdot,\cdot)\) is indefinite and non-Hermitian due to the negative mass term and the imaginary impedance boundary contribution, it is not suitable for defining a stable local spectral decomposition. The form \(a_i(\cdot,\cdot)\) is instead designed as a positive auxiliary energy. In particular, the mass term $k^2 \int_{K_i} \mathbf{w} \cdot \overline{\mathbf{v}}dx$ ($k>0$) is taken with a positive sign and the impedance boundary contribution is replaced by a positive Silver--Müller-type boundary penalty $k \int_{\partial K_i \cap \partial \Omega} \mathbf{w}_T \cdot \overline{\mathbf{v}}_Tds$. 
This choice provides control of \(\|\operatorname{curl}\mathbf v\|_{L^2(K_i)}\), \(\|\mathbf v\|_{L^2(K_i)}\), and the tangential boundary components on \(\partial K_i\cap\partial\Omega\), which guarantees the coercivity of the local spectral problem. The resulting eigenvalue problem provides an auxiliary space in which
curl-kernel components are automatically excluded from the dominant eigenspaces and no explicit divergence-free constraint needs to be imposed on the local auxiliary multiscale basis functions. Thus, the purpose of (\ref{local spectral problem})-(\ref{def of a and s}) is to select robust auxiliary modes for the subsequent multiscale construction, rather than to approximate the spectrum of the original Maxwell operator.}
\end{remark}

{The influence of the mass term can also be understood from the Rayleigh quotient of the local spectral problem,
\[
\lambda_i(\mathbf v)=
\frac{
\displaystyle
\int_{K_i}\mu_r^{-1}|\operatorname{curl}\mathbf v|^2dx
+
k^2\int_{K_i}|\mathbf v|^2dx
+
k\int_{\partial K_i\cap\partial\Omega}|\mathbf v_T|^2ds
}{
\displaystyle
\int_{K_i}\mu_r^{-1}H^{-2}|\mathbf v|^2dx
}.
\]
Without the positive mass term \(k^2\int_{K_i}|\mathbf v|^2dx\), the auxiliary energy would not control nonzero fields in the kernel of the curl operator. In particular, for interior coarse blocks where \(\partial K_i\cap\partial\Omega=\emptyset\), such fields could have zero curl energy and may pollute the dominant eigenspaces. The mass term assigns nonzero auxiliary energy to these curl-kernel components and therefore suppresses such spurious modes in the spectral selection process. Meanwhile, the heterogeneity information is still retained through the coefficient-weighted curl term in \(a_i(\cdot,\cdot)\) and the weighted \(s_i(\cdot,\cdot)\)-inner product. Thus, the local eigenfunctions continue to capture the dominant multiscale features associated with the high-contrast coefficient, while the mass term regularizes the curl-kernel components and the Silver--Müller-type boundary penalty controls the tangential trace on the physical boundary. Computationally, this construction avoids the need to impose explicit divergence-free constraints, introduce Lagrange multipliers, or solve local saddle-point eigenvalue problems, and hence simplifies the local basis generation procedure.}

We denote norms related to bilinear forms $a_i(\cdot,\cdot)$ and $s_i(\cdot,\cdot)$
$$
\begin{aligned}
\|\mathbf{w}\|_{a_i}^2 \coloneqq \|\mathbf{w}\|_{a(K_i)}^2 =\left( \mu_r^{-1} \operatorname{curl} \mathbf{w}, \operatorname{curl} \mathbf{w} \right)_{K_i} + k^2 \left( \mathbf{w}, \mathbf{w} \right)_{K_i}  + k \left\langle  \mathbf{w}_T, \mathbf{w}_T \right\rangle_{\partial K_i \cap \partial \Omega},
\end{aligned}
$$
and
$$
\begin{gathered}
\|\mathbf{w}\|_a^2 = \sum_{i=1}^N \|\mathbf{w}\|_{a_i}^2 = \left( \mu_r^{-1} \operatorname{curl}\mathbf{w}, \operatorname{curl} \mathbf{w} \right)
+ k^2\left( \mathbf{w}, \mathbf{w} \right) + k\left\langle  \mathbf{w}_T, \mathbf{w}_T \right\rangle.
\end{gathered}
$$
Clearly, we have the following equivalence for norms $\norm{\cdot}_a$ and $\norm{\cdot}_{k,imp}$
\begin{equation}
\label{equivalence for two norms}
\min\{\mu_\mathup{max}^{-1}, 1\} \|\mathbf{w}\|_{k, \mathrm{imp}}^2 \leq \|\mathbf{w}\|_a^2 \leq \max \{\mu_\mathup{min}^{-1}, 1\} \|\mathbf{w}\|_{k, \mathrm{imp}}^2,
\end{equation}
Define $\|\mathbf{w}\|_{s_i}^2\coloneqq\|\mathbf{w}\|_{s(K_i)}^2 = \left( \mu_r^{-1} H^{-2} \mathbf{w}, \mathbf{w} \right)_{K_i}$ and then
$$
\|\mathbf{w}\|_s^2 = \sum_{i=1}^N \|\mathbf{w}\|_{s_i}^2 = \left( \mu_r^{-1} H^{-2} \mathbf{w}, \mathbf{w} \right).
$$
Let the eigenvalues $\lambda^i_j$ be in ascending order:

$$
0 < \lambda^i_1 \leqslant \lambda^i_2 \leqslant \cdots \leqslant \lambda^i_{l_i} \leqslant \lambda^i_{l_i+1} \leqslant \cdots,
$$
and we use the first $l_i$ eigenfunctions to construct the local auxiliary space $\bm{V}_{\mathup{aux}}^i = \left\{\bm{\phi}^i_1, \bm{\phi}^i_2, \cdots, \bm{\phi}^i_{l_i}\right\}$.
The global auxiliary space $\bm{V}_{\mathup{aux}}$ is the sum of these local auxiliary spaces, namely $\bm{V}_{\mathup{aux}} = \bigoplus_{i=1}^N \bm{V}_{\mathup{aux}}^i$, which will be used to construct multiscale basis functions. Using the inner product defined above in the eigenproblem, we can define the notion of $\bm{\phi}_j^i$-orthogonality. For a given function $\bm{\phi}_j^i \in \bm{V}_{\mathup{aux}}$, we say that a function $\bm{\psi} \in \bm{V}$ is $\bm{\phi}_j^i$-orthogonal if $s\left(\bm{\phi}_j^i, \bm{\psi}\right) = 1$, and $s\left(\bm{\phi}_{j'}^{i'}, \bm{\psi}\right) = 0$ if $j' \neq j$ or $i' \neq i$. we assume the normalization $s_i\left(\bm{\phi}_j^i, \bm{\phi}_j^i\right) = 1$.
The orthogonal projection $\pi_i$ from $\bm{L}^2\left(K_i\right)$ onto $\bm{V}_{\mathup{aux}}^i$ is then defined by
$$
\pi_i(\mathbf{v}) := \sum_{j=1}^{l_i} s_i\left(\bm{\phi}_j^i, \mathbf{v}\right) \bm{\phi}_j^i, \quad \forall \mathbf{v} \in \bm{L}^2\left(K_i\right).
$$
In addition, we let $\pi: \bm{L}^2(\Omega) \rightarrow \bm{V}_{\mathup{aux}}$ be the projection with respect to the inner product $s(\mathbf{v},\mathbf{w})$. So, the operator $\pi$ is given by $\pi(\mathbf{v}) = \sum_{i=1}^N \sum_{j=1}^{l_i} s_i\left(\bm{\phi}_j^i, \mathbf{v}\right) \bm{\phi}_j^i,$ for all $\mathbf{v} \in \bm{L}^2(\Omega)$. Note that $\pi = \sum_{i=1}^N \pi_i$.

The following Lemma 1 demonstrates the properties of the global projection $\pi$, which will be frequently utilized in the analysis. Its proof is straightforward based on the local spectral problem (\ref{local spectral problem})-(\ref{def of a and s}).
\begin{lemma}\label{property of pi}
In each $K_i \in \mathcal{T}_H$, for all $\mathbf{v} \in \bm{V}\left(K_i\right)$, we have 
$$
\left\|\mathbf{v} - \pi_i (\mathbf{v})\right\|_{s_i}^2 \leqslant \Lambda^{-1} \|\mathbf{v}\|_{a_i}^2,
$$
where $\Lambda = \min_{1 \leqslant i \leqslant N} \lambda_{l_i+1}^i$, and
$$
\left\|\pi_i(\mathbf{v})\right\|_{s_i}^2 \leq \|\mathbf{v}\|_{s_i}^2.
$$
\end{lemma}
\subsection{Multiscale basis functions} \label{Multiscale basis}

Since the operator $B$ (see (\ref{weak form})) is not Hermitian, it is necessary to define two bounded operators to proceed. Specifically, we denote $T_m = \sum_{i=1}^N T_{i, m}$ and $T_m^* = \sum_{i=1}^N T_{i, m}^*$ both from $\bm{L}^2$ to $\bm{V}$ to construct the multiscale trial space and multiscale test space. For each coarse element $K_i \in \mathcal{T}_H$ and its oversampled domain $K_{i, m} \subset \Omega$, we define the multiscale basis functions $T_{i, m} \bm{\phi}_j^i \in \bm{V}_0(K_{i, m})$ (here $\bm{V}_0(K_{i, m})$ is the subspace of $\bm{V}(K_{i, m})$ with zero tangential trace on $\partial K_{i, m}$; this choice is to ensure the conforming property of the $\bm{H}(\mathup{curl})$-conforming bases). Find $T_{i, m} \bm{\phi}_j^i \in \bm{V}_0(K_{i, m})$ such that
\begin{equation}
\label{variational form for ms}
B\left(T_{i, m} \bm{\phi}_j^i, \mathbf{v}\right) + s\left(\pi (T_{i, m} \bm{\phi}_j^i), \pi \mathbf{v}\right) = s\left(\bm{\phi}_j^i, \pi \mathbf{v}\right) \quad \forall \mathbf{v} \in \bm{V}_0(K_{i, m}).
\end{equation}
Our multiscale finite element space $\bm{V}_{\mathup{ms}}$ can be defined by solving the variational problem (\ref{variational form for ms}):
$$
\bm{V}_{\mathup{ms}} = \operatorname{span}\left\{T_{i, m} \bm{\phi}_j^i \colon 1 \leq i \leq N, 1 \leq j \leq l_i\right\}.
$$
The global multiscale basis function $T_i \bm{\phi}_j^i$ is defined similarly,
\begin{equation}
\label{variational form for global ms}
B\left(T_i \bm{\phi}_j^i, \mathbf{v}\right) + s\left(\pi (T_i \bm{\phi}_j^i), \pi \mathbf{v}\right) = s\left(\bm{\phi}_j^i, \pi \mathbf{v}\right), \quad \forall \mathbf{v} \in \bm{V}. 
\end{equation}
Then the global multiscale finite element space $\bm{V}_{\mathup{glo}}$ is defined by
$$
\bm{V}_{\mathup{glo}} = \operatorname{span}\left\{T_i \bm{\phi}_j^i \colon 1 \leq i \leq N, 1 \leq j \leq l_i\right\}.
$$
Similarly, for the local adjoint operator $T_{i, m}^*$ from $\bm{L}^2$ to $\bm{V}_0(K_{i, m})$,
\begin{equation}
\label{variational form for ms dual}
B\left(\mathbf{v}, T_{i, m}^* \bm{\phi}_j^i\right) + s\left(\pi \mathbf{v}, \pi (T_{i, m}^* \bm{\phi}_j^i)\right) = s\left(\pi \mathbf{v}, \bm{\phi}_j^i\right) \quad \forall \mathbf{v} \in \bm{V}_0(K_{i, m}),
\end{equation}
where $T_{i, m}^* \bm{\phi}_j^i = \overline{T_{i, m} \bm{\phi}_j^i}$. Now, another multiscale finite element space $\bm{V}_{\mathup{ms}}^*$ can be defined by solving (\ref{variational form for ms dual}):
$$
\bm{V}_{\mathup{ms}}^* = \operatorname{span}\left\{T_{i, m}^* \bm{\phi}_j^i \colon 1 \leq i \leq N, 1 \leq j \leq l_i\right\}.
$$
The global multiscale basis function $T_i^* \bm{\phi}_j^i \in \bm{V}$ is defined similarly,
\begin{equation}
\label{variational form for global ms dual}
B\left(\mathbf{v}, T_i^* \bm{\phi}_j^i\right) + s\left(\pi \mathbf{v}, \pi(T_i^* \bm{\phi}_j^i)\right) = s\left(\pi \mathbf{v}, \bm{\phi}_j^i\right), \quad \forall \mathbf{v} \in \bm{V}.
\end{equation}
Therefore, another global multiscale finite element space $\bm{V}_{\mathup{glo}}^*$ is defined by
$$
\bm{V}_{\mathup{glo}}^* = \operatorname{span}\left\{T_i^* \bm{\phi}_j^i \colon 1 \leq i \leq N, 1 \leq j \leq l_i\right\}.
$$
And we define $T=\sum_{i=1}^N T_i, T^*=\sum_{i=1}^N T_i^*$. The well-posedness of (\ref{variational form for ms})-(\ref{variational form for global ms dual}) will be proved by the coercivity of $B(\cdot, \cdot) + s(\pi(\cdot), \pi(\cdot))$ below. In the following, we use $\bm{V}_{\mathup{ms}}$ and $\bm{V}_{\mathup{ms}}^*$ as the new trial space and test space of the Petrov-Galerkin framework to find the approximated solution of (\ref{weak form}): find $\mathbf{u}_{\mathup{ms}} \in \bm{V}_{\mathup{ms}}$ such that
\begin{equation}
\label{multiscale discretization}
B\left(\mathbf{u}_{\mathup{ms}}, \mathbf{v}\right) = (\mathbf{f}, \mathbf{v}), \quad \forall \mathbf{v} \in \bm{V}_{\mathup{ms}}^*. 
\end{equation}
{Although the auxiliary eigenfunctions in (\ref{local spectral problem})-(\ref{def of a and s}) may be chosen real-valued, the resulting multiscale approximation is complex-valued. Indeed, the local multiscale basis functions \(T_{i,m}\boldsymbol{\phi}_j^i\) are computed from the complex-valued variational problem (\ref{variational form for ms}), which involves the original Maxwell sesquilinear form \(B(\cdot,\cdot)\). Hence \(T_{i,m}\boldsymbol{\phi}_j^i\) is generally complex-valued even if \(\boldsymbol{\phi}_j^i\) is real-valued. The final multiscale solution is represented as
\[
\mathbf{u}_{\mathrm{ms}}=
\sum_{i=1}^N\sum_{j=1}^{l_i}
c_j^i T_{i,m}\boldsymbol{\phi}_j^i,
\qquad c_j^i\in\mathbb C.
\]
This complex span is sufficient to approximate the complex-valued Maxwell solution, and no additional complex correction or separate real-imaginary decomposition is required.}
Given a function $\widetilde{\mathbf{v}} \in \widetilde{\bm{V}} \coloneqq \{\mathbf{v} \in \bm{V} \colon \pi(\mathbf{v}) = 0\}$, we have $\pi(\widetilde{\mathbf{v}}) = \mathbf{0}$.

Based on the construction of the global multiscale space, we define the global problem as follows: find $\mathbf{u}_{\mathup{glo}} \in \bm{V}_{\mathup{glo}}$ such that
\begin{equation}
\label{global problem}
B\left(\mathbf{u}_{\mathup{glo}}, \mathbf{v}\right) = (\mathbf{f}, \mathbf{v}), \quad \forall \mathbf{u} \in \bm{V}_{\mathup{glo}}^*.
\end{equation}
\section{Analysis}\label{sec:ana}
\begin{assumption}
\label{resolution condition}
Suppose the coarse mesh size $H$, the wave number $k$, the relative permeability, 
and $\Lambda$ satisfy the following resolution condition:
$$
 kH\sqrt{\mu_{\mathup{max}}}\Lambda^{-1/2}< \sqrt{\frac{\widetilde{c}}{2}},
$$
where $0 < \widetilde{c} \ll 1$. We also suppose there exists $0< c_0< 1$ such that $2\widetilde{c}\leq 2\widetilde{c}\Lambda \leq c_0 < 1$ for the convenience of the following analysis. Note that $\Lambda=\min_{1\leq i\leq N}\lambda_{l_i+1}^i$ maintains independent of $h,H$ and the relative permeability $\mu_r$ \cite{Galvis2010}.
\end{assumption}
{The above resolution condition ensures that the negative mass contribution in \(B(\cdot,\cdot)\) can be controlled by the positive auxiliary energy and the \(s\)-projection term, which is essential for the coercivity estimate in the following Lemma \ref{covercity}. Equivalently, the condition can be viewed as
\(
kH \leq \sqrt{\frac{\widetilde{c}\Lambda}{2\mu_{\max}}}.
\)
Thus, higher wave numbers or larger permeability contrast require either a finer coarse mesh or a larger spectral gap \(\Lambda\). Since \(\Lambda\) is the first neglected eigenvalue in the local auxiliary spectral problems, it can be increased by including more auxiliary basis functions. Therefore, unlike a purely mesh-based resolution condition, the present condition can also be relaxed through spectral enrichment. Such resolution assumptions are common in the analysis of high-frequency Helmholtz \cite{Peterseim2017} and time-harmonic Maxwell problems, where stability requires that the coarse space sufficiently resolves the effective wavelength and the relevant multiscale features. In the present CEM-type framework, the spectral enrichment mechanism provides an additional way to improve this condition beyond simply refining the coarse mesh.}

\begin{lemma}
\label{covercity}
Under Assumption \ref{resolution condition}, there exists $0 < \alpha < 1$ independent of $k,\mu_r,H,\Lambda$ such that
$$
\left| B(\mathbf{v}, \mathbf{v}) + s(\pi\mathbf{v}, \pi\mathbf{v}) \right| \geq \alpha \left[ \|\mathbf{v}\|_a^2 + \|\pi\mathbf{v}\|_s^2 \right]
$$
for all $\mathbf{v} \in \bm{V}$. Consequently, we also have
$$
\left| B(\widetilde{\mathbf{v}}, \widetilde{\mathbf{v}}) \right| \geq \alpha \|\widetilde{\mathbf{v}}\|_a^2 \quad \forall \widetilde{\mathbf{v}} \in \widetilde{\bm{V}}.
$$
\end{lemma}
\begin{proof}
Note that
$$
\begin{aligned}
& B(\mathbf{v}, \mathbf{v}) + s(\pi\mathbf{v}, \pi\mathbf{v}) 
=  \left( \mu_r^{-1} \operatorname{curl} \mathbf{v}, \operatorname{curl} \mathbf{v} \right) - k^2 (\mathbf{v}, \mathbf{v}) - i k \left\langle \mathbf{v}_T, \mathbf{v}_T \right\rangle + s(\pi\mathbf{v}, \pi\mathbf{v}) \\
= & \left[ \left( \mu_r^{-1} \operatorname{curl} \mathbf{v}, \operatorname{curl} \mathbf{v} \right) + k^2 (\mathbf{v}, \mathbf{v}) + k \left\langle \mathbf{v}_T, \mathbf{v}_T \right\rangle + s(\pi\mathbf{v}, \pi\mathbf{v}) \right]- 2k^2 (\mathbf{v}, \mathbf{v}) - k \left\langle \mathbf{v}_T, \mathbf{v}_T \right\rangle - i k \left\langle \mathbf{v}_T, \mathbf{v}_T \right\rangle \\
= & \left[ a(\mathbf{v}, \mathbf{v}) + s(\pi\mathbf{v}, \pi\mathbf{v}) \right] - 2k^2 (\mathbf{v}, \mathbf{v}) - k \left\langle \mathbf{v}_T, \mathbf{v}_T \right\rangle - i k \left\langle \mathbf{v}_T, \mathbf{v}_T \right\rangle.
\end{aligned}
$$
First, we prove that $-2k^2(\mathbf{v}, \mathbf{v})$ can be controlled by $a(\mathbf{v}, \mathbf{v}) + s(\pi\mathbf{v}, \pi\mathbf{v})$. By the definition of $s$-norm and Lemma \ref{property of pi}, we have
$$
\begin{aligned}
2k^2 \|\mathbf{v}\|_{L^2}^2 & \leq 2k^2\max_{\mathbf{x}\in\Omega} (\mu_r(\mathbf{x}))H^2 \|\mathbf{v}\|_s^2 \leq 2k^2 \mu_{\mathup{max}} H^2\cdot 2 \left( \|\mathbf{v}-\pi\mathbf{v}\|_s^2 + \|\pi\mathbf{v}\|_s^2 \right) \\
& \leq 2k^2 \mu_{\mathup{max}} H^2 \cdot 2\left( \|\pi\mathbf{v}\|_s^2 + \frac{1}{\Lambda} \|\mathbf{v}\|_a^2 \right) = 4\mu_{\mathup{max}} k^2 H^2\|\pi\mathbf{v}\|_s^2 + 4\mu_{\mathup{max}} k^2 H^2\Lambda^{-1} \|\mathbf{v}\|_a^2.
\end{aligned}
$$
By using the resolution condition in Assumption \ref{resolution condition}, we have
$$
\begin{aligned}
4\mu_{\mathup{max}}k^2 H^2 < 2\Lambda \widetilde{c} \leq c_0 < 1, \quad
4\mu_{\mathup{max}} k^2 H^2\Lambda^{-1} < 2\widetilde{c} \leq c_0 < 1.
\end{aligned}
$$
Thus we can obtain that $-2k^2(\mathbf{v}, \mathbf{v}) \geq -c_0 \left( \|\pi\mathbf{v}\|_s^2 + \|\mathbf{v}\|_a^2 \right)$. Then we have
$$
\begin{aligned}
a(\mathbf{v}, \mathbf{v}) + s(\pi\mathbf{v}, \pi\mathbf{v}) - 2k^2(\mathbf{v}, \mathbf{v}) \geq \beta \left[ \|\mathbf{v}\|_a^2 + \|\pi\mathbf{v}\|_s^2 \right],
\end{aligned}
$$
where $\beta = 1 - c_0$, independent of $k$, $H$, $\mu_r$, $\Lambda$.

We clearly see that (For notational convenience, denote $T_1(\mathbf{v}, \mathbf{v})\coloneqq a(\mathbf{v}, \mathbf{v}) + s(\pi\mathbf{v}, \pi\mathbf{v}) - 2k^2(\mathbf{v}, \mathbf{v}) $):
$$
\begin{aligned}
& \left| B(\mathbf{v}, \mathbf{v}) + s(\pi\mathbf{v}, \pi\mathbf{v}) \right|^2 = \left| T_1(\mathbf{v}, \mathbf{v}) - k\langle \mathbf{v}_T, \mathbf{v}_T\rangle - ik\langle \mathbf{v}_T, \mathbf{v}_T\rangle \right|^2 \\
= & \left[T_1(\mathbf{v}, \mathbf{v}) - k\langle \mathbf{v}_T, \mathbf{v}_T\rangle \right]^2 + \left[ k\langle \mathbf{v}_T, \mathbf{v}_T\rangle \right]^2 \\
= & \left[T_1(\mathbf{v}, \mathbf{v}) \right]^2 + 2k^2 \|\mathbf{v}_T\|_{L^2}^2 - \left[T_1(\mathbf{v}, \mathbf{v})\right] \cdot 2k \|\mathbf{v}_T\|_{L^2}^2\\
\geq & \left[T_1(\mathbf{v}, \mathbf{v})\right]^2 + 2k^2 \|\mathbf{v}_T\|_{L^2}^2 - \left( \frac{1}{2} \left[T_1(\mathbf{v}, \mathbf{v})\right]^2 + 2k^2 \|\mathbf{v}_T\|_{L^2}^2 \right) \\
= & \frac{1}{2} \left[T_1(\mathbf{v}, \mathbf{v})\right]^2
\geq \frac{1}{2} \beta^2 \left[ \|\mathbf{v}\|_a^2 + \|\pi\mathbf{v}\|_s^2 \right]^2
\end{aligned}
$$
Therefore, we obtain
\[
\left| B(\mathbf{v}, \mathbf{v}) + s(\pi\mathbf{v}, \pi\mathbf{v}) \right| \geq \alpha \left[ \|\mathbf{v}\|_a^2 + \|\pi\mathbf{v}\|_s^2 \right] \quad \forall \mathbf{v} \in \bm{V},
\]
where $\alpha = \frac{\sqrt{2}}{2} \beta$, independent of $k$, $H$, $\mu_r$, $\Lambda$.

For any $\widetilde{\mathbf{v}} \in \widetilde{\bm{V}}$, we have $\pi(\widetilde{\mathbf{v}}) = 0$ by definition, then it's clear that
\[
\left| B(\widetilde{\mathbf{v}}, \widetilde{\mathbf{v}}) \right| \geq \alpha \|\widetilde{\mathbf{v}}\|_a^2,\quad \forall \widetilde{\mathbf{v}} \in \widetilde{\bm{V}}.
\]
\end{proof}

Then we give the well-posedness for the global problem in Theorem \ref{global well-posedness} as follows.

\begin{theorem}
\label{global well-posedness}
The bilinear form $B$ satisfies the following inf-sup condition: there exists $\beta(k, \mu_{\mathup{min}}) > 0$ such that

\[
\inf_{\mathbf{v} \in \bm{V}_{\mathup{glo}} \setminus \{\mathbf{0}\}} \sup_{\mathbf{v}^* \in \bm{V}_{\mathup{glo}}^* \setminus \{\mathbf{0}\}} \frac{\left| B(\mathbf{v}, \mathbf{v}^*) \right|}{\|\mathbf{v}\|_{k, \mathrm{imp}} \cdot \|\mathbf{v}^*\|_{k, \mathrm{imp}}} \geq \beta(k, \mu_{\mathup{min}}) > 0,
\]
where $\beta(k, \mu_{\mathup{min}}) = \frac{1}{2C_k\max\{\sqrt{k},k\}} \left( 1 + \frac{1}{\alpha}\max \left\{\mu_{\min}^{-1}, 1 \right\} \right)^{-1} $.
\end{theorem}
\begin{proof}
In term of Lemma \ref{covercity}, we know $\left| B(\cdot, \cdot) + s(\pi(\cdot), \pi(\cdot)) \right|$ is an inner product. By (\ref{variational form for global ms}) and the definition of $\widetilde{\bm{V}}$, we have $B\left(T_i \bm{\phi}_j^i, \widetilde{\mathbf{v}}\right)+s\left(\pi (T_i \bm{\phi}_j^i), \pi \widetilde{\mathbf{v}}\right) = 0$, that is, with respect to the inner product $\left| B(\cdot, \cdot) + s(\pi(\cdot), \pi(\cdot)) \right|$, $\widetilde{\bm{V}}\subset \bm{V}_{\mathup{glo}}^\perp$. Since $\dim(\bm{V}_{\mathup{glo}}) = \dim(\bm{V}_{\mathup{aux}})$, we have $\widetilde{\bm{V}} = \bm{V}_{\mathup{glo}}^\perp$. Thus we have that $\mathbf{v} = \widetilde{\bm{V}} \oplus \bm{V}_{\mathup{glo}}$ under the inner product $\left| B(\cdot, \cdot) + s(\pi(\cdot), \pi(\cdot)) \right|$. Similarly, we also have $\bm{V} = \widetilde{\bm{V}} \oplus \bm{V}_{\mathup{glo}}^*.$ Note that $\bm{V}_{\mathup{glo}}\neq\bm{V}_{\mathup{glo}}^*$. Clearly, by (\ref{variational form for global ms}) and (\ref{variational form for global ms dual}), we have $B\left(\mathbf{v}_1, \widetilde{\mathbf{v}}\right) = 0$, $B\left(\widetilde{\mathbf{v}}, \mathbf{v}_1^*\right) = 0$ for all $\mathbf{v}_1 \in \bm{V}_{\mathup{glo}}$, $\mathbf{v}_1^* \in \bm{V}_{\mathup{glo}}^*$, $\widetilde{\mathbf{v}} \in \widetilde{\bm{V}}$.

By (\ref{continuous inf-sup}) and $\bm{V} = \bm{V}_{\mathup{glo}}^* \oplus \widetilde{\bm{V}}$, we know
$\forall \mathbf{v} \in \bm{V}_{\mathup{glo}} \subset \bm{V}$, there exist $\mathbf{v}^* \in \bm{V}_{\mathup{glo}}^*$, $\widetilde{\mathbf{v}} \in \widetilde{\bm{V}}$ such that
\[
\left| B(\mathbf{v}, \mathbf{v}^* + \widetilde{\mathbf{v}}) \right| \geq \frac{1}{2C_k\max\{\sqrt{k},k\}} \|\mathbf{v}\|_{k, \mathrm{imp}} \cdot \|\mathbf{v}^* + \widetilde{\mathbf{v}}\|_{k, \mathrm{imp}}.
\]
Since $B(\mathbf{v}, \widetilde{\mathbf{v}}) = 0$ for $\widetilde{\mathbf{v}} \in \widetilde{\bm{V}}$ and $\mathbf{v} \in \bm{V}_{\mathup{glo}}$, then

\[
\left| B(\mathbf{v}, \mathbf{v}^*) \right| \geq \frac{1}{2C_k\max\{\sqrt{k},k\}} \|\mathbf{v}\|_{k, \mathrm{imp}} \cdot \|\mathbf{v}^*+ \widetilde{\mathbf{v}}\|_{k, \mathrm{imp}}.
\]
Then we aim to show $\|\mathbf{v}^*\|_{k, \mathrm{imp}} \leq c\|\mathbf{v}^* + \widetilde{\mathbf{v}}\|_{k, \mathrm{imp}}$.
Since $\|\mathbf{v}^*\|_{k, \mathrm{imp}} \leq \|\mathbf{v}^* + \widetilde{\mathbf{v}}\|_{k, \mathrm{imp}} + \|\widetilde{\mathbf{v}}\|_{k, \mathrm{imp}}$, $B( \widetilde{\mathbf{v}},\mathbf{v}^*) = 0$, and
by Lemma \ref{covercity}, we have
\[
\begin{aligned}
\|\widetilde{\mathbf{v}}\|_{k, \mathrm{imp}}^2 & \leq \frac{1}{\alpha} \left| B(\widetilde{\mathbf{v}}, \widetilde{\mathbf{v}}) \right| = \frac{1}{\alpha} \left| B(\widetilde{\mathbf{v}}, \mathbf{v}^* + \widetilde{\mathbf{v}}) \right|  \leq \frac{1}{\alpha} \max \left\{\mu_{\min}^{-1}, 1 \right\} \|\widetilde{\mathbf{v}}\|_{k, \mathrm{imp}} \cdot \|\mathbf{v}^*+\widetilde{\mathbf{v}}\|_{k, \mathrm{imp}}.
\end{aligned}
\]
Thus we have
\[
\|\mathbf{v}^*\|_{k, \mathrm{imp}} \leq \left( 1 + \frac{1}{\alpha}\max \left\{\mu_{\min}^{-1}, 1 \right\} \right) \|\mathbf{v}^* + \widetilde{\mathbf{v}}\|_{k, \mathrm{imp}}.
\]
That is, we obtain that $\left| B(\mathbf{v}, \mathbf{v}^*) \right| \geq \frac{1}{2C_k\max\{\sqrt{k},k\}} \left( 1 + \frac{1}{\alpha}\max \left\{\mu_{\min}^{-1}, 1 \right\} \right)^{-1} \|\mathbf{v}\|_{k, \mathrm{imp}} \cdot \|\mathbf{v}^*\|_{k, \mathrm{imp}}$.
Therefore, we get
\[
\inf_{\mathbf{v} \in \bm{V}_{\mathup{glo}} \setminus \{\mathbf{0}
\}} \sup_{\mathbf{v}^* \in \bm{V}_{\mathup{glo}f}^* \setminus \{\mathbf{0}\}} \frac{\left| B(\mathbf{v}, \mathbf{v}^*) \right|}{\|\mathbf{v}\|_{k, \mathrm{imp}} \cdot \|\mathbf{v}^*\|_{k, \mathrm{imp}}} \geq\beta(k, \mu_{\mathup{min}})  > 0,
\]
where $\beta(k, \mu_{\mathup{min}}) = \frac{1}{2C_k\max\{\sqrt{k},k\}} \left( 1 + \frac{1}{\alpha}\max \left\{\mu_{\min}^{-1}, 1 \right\} \right)^{-1} $.
The well-posedness of the global problem has been proved.
\end{proof}
\begin{theorem}
\label{global convergence}
Under Assumption \ref{resolution condition}, let $\mathbf{u}_{\mathup{glo}}\in\bm{V}_{\mathup{glo}}$ be the solution of (\ref{global problem}) and $\mathbf{u}\in\bm{V}$ be the solution of (\ref{weak form}). Then
\[
\left\| \mathbf{u}_{\mathup{glo}} - \mathbf{u} \right\|_a \leq \frac{1}{\alpha \sqrt{\Lambda}} \|\mathbf{f}\|_{s^{-1}},
\]
where $\alpha > 0$ is a constant independent of $k$, $H$, $\mu_r$, $\Lambda$ from Lemma \ref{covercity}.
\end{theorem}
\begin{proof}
By (\ref{weak form}) and (\ref{global problem}), we directly have
\[
B(\mathbf{u} - \mathbf{u}_{\mathup{glo}}, \mathbf{v}) = 0 \quad \forall \mathbf{v} \in \bm{V}_{\mathup{glo}}^*.\]
In terms of the orthogonality property, we know $\mathbf{u} - \mathbf{u}_{\mathup{glo}} \in \widetilde{\bm{V}}$, i.e. $\pi(\mathbf{u} - \mathbf{u}_{\mathup{glo}}) = 0$ by definition. By the orthogonality $B(\mathbf{u}_{\mathup{glo}}, \mathbf{u} - \mathbf{u}_{\mathup{glo}})=0$, and Lemmas \ref{property of pi}, \ref{covercity}, we obtain
\[
\begin{aligned}
\left\| \mathbf{u} - \mathbf{u}_{\mathup{glo}} \right\|_a^2 & \leq \frac{1}{\alpha} \left| B(\mathbf{u} - \mathbf{u}_{\mathup{glo}}, \mathbf{u} - \mathbf{u}_{\mathup{glo}}) \right|  = \frac{1}{\alpha} \left| (f, \mathbf{u} - \mathbf{u}_{\mathup{glo}}) \right| \\
& \leq \frac{1}{\alpha} \|\mathbf{f}\|_{s^{-1}} \cdot \left\| \mathbf{u} - \mathbf{u}_{\mathup{glo}} \right\|_s  \leq \frac{1}{\alpha} \|\mathbf{f}\|_{s^{-1}} \cdot \frac{1}{\sqrt{\Lambda}} \left\| \mathbf{u} - \mathbf{u}_{\mathup{glo}} \right\|_a,
\end{aligned}
\]
Thus we have $\left\| \mathbf{u} - \mathbf{u}_{\mathup{glo}} \right\|_a \leq \frac{1}{\alpha \sqrt{\Lambda}} \|\mathbf{f}\|_{s^{-1}}.$
\end{proof}

By the preceding theorem, we obtain convergence of the method when global basis functions are employed. We shall now prove that these global basis functions admit a localizable construction. First, we establish a lemma that will later be used to estimate the difference between the global and the multiscale basis functions.
For each coarse block~$K\in\mathcal{T}_H$, we introduce a bubble function~$B$ such that $B(\mathbf{x})>0$ for every $\mathbf{x}$ in the interior of~$K$, and $B(\mathbf{x})=0$ for all~$\mathbf{x}\in\partial K$. In particular, we take $B=\prod_{j}\eta_{j}$, where $\{\eta_{j}\}_j$ is the set of Lagrange basis functions on the coarse element $K$ and the product extends over all vertices~$j$ on the boundary of~$K$. Using this bubble function we define the constant
\[
C_{\pi}= \sup_{K\in\mathcal{T}_{H},\;\mathbf{v}\in \bm{V}_{\mathrm{aux}}}
          \frac{\displaystyle\int_{K}H^{-2}\mu_r^{-1}\,\mathbf{v}\cdot\overline{\mathbf{v}}\,dx}{\displaystyle\int_{K}B(\mathbf{x})\,H^{-2}\mu_r^{-1}\,\mathbf{v}\cdot\overline{\mathbf{v}}\,dx}.
\]

\begin{lemma}
\label{aux_lemma}
For all $\mathbf{v}_{\mathup{aux}} \in \bm{V}_{\mathup{aux}}$, there exists $\mathbf{v} \in \bm{V}$ such that
\[
\pi(\mathbf{v}) = \mathbf{v}_{\mathup{aux}}, \quad \|\mathbf{v}\|_a^2 \leq D \|\mathbf{v}_{\mathup{aux}}\|_s^2, \quad \operatorname{supp}(\mathbf{v}) \subset \operatorname{supp}(\mathbf{v}_{\mathup{aux}}),
\]
where $D$ does not depend on $H,k,\mu_r$, but depends on the eigenvalue $\max_{\substack{1 \leq i \leq N \\ 1 \leq j \leq l_i}} \lambda_j^i$.
\end{lemma}
\begin{proof}
Consider the following minimization problem defined on a coarse block $K_i$:
\begin{equation}
\label{minimization problem}
\mathbf{v} = \operatorname{argmin} \left\{ a(\bm{\psi}, \bm{\psi}) \colon \bm{\psi} \in \bm{V}(K_i), \quad s_i(\bm{\psi}, \mathbf{v}_{\mathup{aux}}) = 1, \quad s_i(\bm{\psi}, \mathbf{w}) = 0 \ \forall \mathbf{w} \in \bm{V}_{\mathup{aux}}^\perp \right\}
\end{equation}
for a given $\mathbf{v}_{\mathup{aux}} \in \bm{V}_{\mathup{aux}}^i$ with $\|\mathbf{v}_{\mathup{aux}}\|_{s_i} = 1$, where $\mathbf{v}_{\mathup{aux}}^\perp \subset \bm{V}_{\mathup{aux}}^i$ is the orthogonal complement of $\operatorname{span}\{\mathbf{v}_{\mathup{aux}}\}$ with respect to the inner product $s_i(\cdot,\cdot)$.
Let $\mathbf{p} \in \bm{V}_{\mathup{aux}}^i$. The minimization problem (\ref{minimization problem}) is equivalent to the following variational problem: Find $\mathbf{v} \in \bm{V}(K_i)$ and $\mathbf{y} \in \bm{V}_{\mathup{aux}}^i$ such that
\begin{subequations}
\label{variational for mini}
\begin{align}
a_i(\mathbf{v}, \mathbf{w}) + s_i(\mathbf{w}, \mathbf{y}) & = 0 \quad \forall \mathbf{w} \in \bm{V}(K_i), \label{variational for mini_a} \\
s_i(\mathbf{v}, \mathbf{z}) & = s_i(\mathbf{p}, \mathbf{z}) \quad \forall \mathbf{z} \in \bm{V}_{\mathup{aux}}^i.\label{variational for mini_b}
\end{align}
\end{subequations}
Note that the well-posedness of the minimization problem (\ref{minimization problem}) is equivalent to the existence of a function $\mathbf{v} \in \bm{V}(K_i)$ such that
\[
s_i(\mathbf{v}, \mathbf{p}) \geq c \|\mathbf{p}\|_{s_i}^2, \quad \|\mathbf{v}\|_{a_i} \leq C \|\mathbf{p}\|_{s_i},
\]
where $C$ is a constant independent of $H,k,\mu_r$. Notice that $\mathbf{p}$ is supported in $K_i$. We let $\mathbf{v} = B(\mathbf{x}) \mathbf{p}$. Then combining the definition of $C_\pi$ and $s$-norm, we have
\[
\begin{aligned}
s_i(\mathbf{v}, \mathbf{p}) = \int_{K_i} \mu_r^{-1} H^{-2} B(\mathbf{x}) \mathbf{p}^2 \, dx  \geq C_\pi^{-1} \|\mathbf{p}\|_{s_i}^2.
\end{aligned}
\]
For $\|\mathbf{v}\|_{a_i}$, we have
\[
\begin{aligned}
\|\mathbf{v}\|_{a_i}^2 = \|B(\mathbf{x}) \mathbf{p}\|_{a_i}^2 = \int_{K_i} \mu_r^{-1} \operatorname{curl}(B \mathbf{p}) \cdot \operatorname{curl}(\overline{B \mathbf{p}})dx + k^2\int_{K_i} B \mathbf{p} \cdot \overline{B \mathbf{p}} \, dx  + k\int_{\partial K_i \cap \partial \Omega} (B \mathbf{p})_T \cdot (\overline{B \mathbf{p}})_T \, ds.
\end{aligned}
\]
Since $\operatorname{curl}(B \mathbf{p}) = (\nabla B) \times \mathbf{p} + B (\operatorname{curl} \mathbf{p})$, then
\[
\begin{aligned}
\left| \operatorname{curl}(B \mathbf{p}) \right|^2 \leq \left[ (\nabla B) \times \mathbf{p} + B (\operatorname{curl} \mathbf{p}) \right]^2 \leq 2 (\nabla B \times \mathbf{p})^2 + 2 B^2 (\operatorname{curl} \mathbf{p})^2.
\end{aligned}
\]
Since for all vectors $\mathbf{w}, \mathbf{v}$, we have $|\mathbf{w} \times \mathbf{v}|^2 = |\mathbf{w}|^2 \cdot |\mathbf{v}|^2 - (\mathbf{w} \cdot \mathbf{v})^2 \leq 2|\mathbf{w}|^2 |\mathbf{v}|^2.$
Then $\left| \operatorname{curl}(B \mathbf{p}) \right|^2 \leq 4|\nabla B|^2 \cdot |\mathbf{p}|^2 + 2|\operatorname{curl} \mathbf{p}|^2$. Note that $|B(\mathbf{x})| \leq 1$, $|\nabla B(\mathbf{x})|^2 \leq 4^2 \sum_j |\nabla \eta_j|^2 \leq c' H^{-2}$ (where $c'>0$ does not depend on $H,k,\mu_r$).
Thus,
\[
\begin{aligned}
&\|\mathbf{v}\|_{a_i}^2 \leq  \int_{K_i} \mu_r^{-1} \left( 2|\operatorname{curl} \mathbf{p}|^2 + 4c'H^{-2}|\mathbf{p}|^2 \right) dx + k^2 \int_{K_i} |\mathbf{p}|^2 dx + k \int_{\partial K_i \cap \partial \Omega} |\mathbf{p}_T|^2 ds \\
& \leq  2\max \left\{2c', 1 \right\} \left( \|\mathbf{p}\|_{s_i}^2 + \|\mathbf{p}\|_{a_i}^2 \right).
\end{aligned}
\]
Because $\|\mathbf{p}\|_{a_i}^2 \leq \left( \max_{1 \leq j \leq l_i} \lambda_j^i \right) \|\mathbf{p}\|_{s_i}^2$, then we have
$
\|\mathbf{v}\|_{a_i}^2 \leq 2\max \left\{2c', 1 \right\} \left( 1 + \max_{1 \leq j \leq l_i} \lambda_j^i \right) \|\mathbf{p}\|_{s_i}^2.
$
That is,
\[
\|\mathbf{v}\|_{a_i} \leq \sqrt{ 2\max \left\{2c', 1 \right\} \left( 1 + \max_{1 \leq j \leq l_i} \lambda_j^i \right) } \|\mathbf{p}\|_{s_i},
\]
which proves the unique solvability of the minimization problem. $(\mathbf{v}, \mathbf{y})$ satisfy (\ref{variational for mini}). From (\ref{variational for mini_b}), we can see that $\pi_i(\mathbf{v}) = \mathbf{p}$. Since $\mathbf{v} = B \mathbf{p}$, it's clear that $\operatorname{supp}(\mathbf{v}) \subset \operatorname{supp}(\mathbf{p})$. By the above estimates, we also have the desired estimate with $D = 2\max \left\{2c', 1 \right\} \left( 1 + \max_{\substack{1 \leq i \leq N \\ 1 \leq j \leq l_i}} \lambda_j^i \right)$. This completes the proof.
\end{proof}
Before estimating the difference between the global and multiscale basis functions, we need some notation and the cutoff function with respect to the oversampling domains. For each $K_i$, we recall that $K_{i,m} \subset \Omega$ is the oversampling coarse region by enlarging $K_i$ by $m$ coarse grid layers. For $M > m$, we define $\chi_i^{M, m} \in \operatorname{span}\left\{\eta_j\right\}$ (recall that $\{\eta_{j}\}_j$ is the set of Lagrange basis functions) such that $0 \leq \chi_i^{M, m} \leq 1$ and
\[
\begin{array}{ll}
\chi_i^{M, m} = 1 & \mathup{in } \quad K_{i, m}, \\
\chi_i^{M, m} = 0 & \mathup{in } \quad \Omega \setminus K_{i, M}.
\end{array}
\]
Note that we have $K_{i, m} \subset K_{i, M}$.

The following theorem shows that our multiscale basis functions have a decay property. In particular, the multiscale basis functions are small outside an oversampled region.

\begin{theorem}\label{decay}
We consider the oversampled domain $K_{i, l}$ with $l \geq 2$. Let $\bm{\phi}_j^i \in \bm{V}_{\mathup{aux}}^i$ be a given auxiliary multiscale basis function. We let $T_{i, l} \bm{\phi}_j^i$ be the multiscale basis functions obtained in (\ref{variational form for ms}) on $K_{i, l}$ and let $T_i \bm{\phi}_j^i$ be the global multiscale functions obtained in (\ref{variational form for global ms}). Then we have
\[
\begin{aligned}
 \left\| T_i \bm{\phi}_j^i - T_{i, l} \bm{\phi}_j^i \right\|_a^2 + \left\| \pi\left( T_i \bm{\phi}_j^i - T_{i, l} \bm{\phi}_j^i \right) \right\|_s^2  \leq \frac{C_*}{\alpha^2} \left( 1 + \frac{1}{\Lambda} \right)\theta^{l-1} \left( \left\| T_i \bm{\phi}_j^i \right\|_a^2 + \left\| \pi\left( T_i \bm{\phi}_j^i \right) \right\|_s^2 \right), 
\end{aligned}
\]
where  $0<\theta=\frac{(1+1/\Lambda)}{(1+1/\Lambda)+\alpha}<1$ and $C_*,\alpha>0$ are independent of $H,k,\mu_r,\Lambda$.
\end{theorem}
\begin{proof}
By the definitions of $T_{i, l} \bm{\phi}_j^i$ and $T_i \bm{\phi}_j^i$ in (\ref{variational form for ms})-(\ref{variational form for global ms}), we have
\[
\begin{aligned}
B\left( T_{i, l} \bm{\phi}_j^i, \mathbf{v} \right) + s\left( \pi\left( T_{i, l} \bm{\phi}_j^i \right), \pi \mathbf{v} \right) & = s\left( \bm{\phi}_j^i, \pi \mathbf{v} \right) \quad \forall \mathbf{v} \in \bm{V}_0(K_{i, l}), \\
B\left( T_i \bm{\phi}_j^i, \mathbf{v} \right) + s\left( \pi\left( T_i \bm{\phi}_j^i \right), \pi \mathbf{v} \right) & = s\left( \bm{\phi}_j^i, \pi \mathbf{v} \right) \quad \forall \mathbf{v} \in \bm{V}.
\end{aligned}
\]
Subtracting the above two equations, we have
\[
B\left(T_i \bm{\phi}_j^i - T_{i,l} \bm{\phi}_j^i, \mathbf{v} \right) + s\left(\pi\left(T_i \bm{\phi}_j^i - T_{i,l} \bm{\phi}_j^i\right), \pi \mathbf{v} \right) = 0
\]
for all $\mathbf{v} \in \bm{V}_0(K_{i,l})$. Taking $\mathbf{v} = \mathbf{w} - T_i \bm{\phi}_j^i + T_i \bm{\phi}_j^i - T_{i,l} \bm{\phi}_j^i$
with $\mathbf{w} \in \bm{V}_0(K_{i,l})$, we have
\begin{equation}
\label{taking test function}
\begin{aligned}
&\quad B\left(T_i \bm{\phi}_j^i - T_{i,l} \bm{\phi}_j^i, T_i \bm{\phi}_j^i - T_{i,l} \bm{\phi}_j^i\right) + s\left(\pi\left(T_i \bm{\phi}_j^i - T_{i,l} \bm{\phi}_j^i\right),\right. \left.\pi\left(T_i \bm{\phi}_j^i - T_{i,l} \bm{\phi}_j^i\right)\right)  \\
& + B\left(T_i \bm{\phi}_j^i - T_{i,l} \bm{\phi}_j^i, \mathbf{w} - T_i \bm{\phi}_j^i\right) + s\left(\pi\left(T_i \bm{\phi}_j^i - T_{i,l} \bm{\phi}_j^i\right), \pi\left(\mathbf{w} - T_i \bm{\phi}_j^i\right)\right) = 0.
\end{aligned}
\end{equation}
By Lemma \ref{covercity}, there exists $\alpha > 0$ independent of $\Lambda$, $k$, $\mu_r$, $H$ such that
\begin{equation}
\label{coarcivity for basis}
\begin{aligned}
& B\left(T_i \bm{\phi}_j^i - T_{i,l} \bm{\phi}_j^i, T_i \bm{\phi}_j^i - T_{i,l} \bm{\phi}_j^i\right) + s\left(\pi\left(T_i \bm{\phi}_j^i - T_{i,l} \bm{\phi}_j^i\right), \pi\left(T_i \bm{\phi}_j^i - T_{i,l} \bm{\phi}_j^i\right)\right) \\
& \geq \alpha \left[ \left\| T_i \bm{\phi}_j^i - T_{i,l} \bm{\phi}_j^i \right\|_a^2 + \left\| \pi\left( T_i \bm{\phi}_j^i - T_{i,l} \bm{\phi}_j^i \right) \right\|_s^2 \right].
\end{aligned}
\end{equation}
Combining (\ref{taking test function}) and (\ref{coarcivity for basis}), we have
\[
\begin{aligned}
& \left\| T_i \bm{\phi}_j^i - T_{i,l} \bm{\phi}_j^i \right\|_a^2 + \left\| \pi\left( T_i \bm{\phi}_j^i - T_{i,l} \bm{\phi}_j^i \right) \right\|_s^2 \\
& \leq \frac{1}{\alpha} \left[ \left| B\left( T_i \bm{\phi}_j^i - T_{i,l} \bm{\phi}_j^i, \mathbf{w} - T_i \bm{\phi}_j^i \right) \right| + \left| s\left( \pi\left( T_i \bm{\phi}_j^i - T_{i,l} \bm{\phi}_j^i \right), \pi\left( \mathbf{w} - T_i \bm{\phi}_j^i \right) \right) \right| \right] \\
& \leq \frac{1}{\alpha} \left[ \left\| T_i \bm{\phi}_j^i - T_{i,l} \bm{\phi}_j^i \right\|_a \cdot \left\| \mathbf{w} - T_i \bm{\phi}_j^i \right\|_a + \left\| \pi\left( T_i \bm{\phi}_j^i - T_{i,l} \bm{\phi}_j^i \right) \right\|_s \cdot \left\| \pi\left( \mathbf{w} - T_i \bm{\phi}_j^i \right) \right\|_s \right] \\
& \leq \frac{ \left\| T_i \bm{\phi}_j^i - T_{i,l} \bm{\phi}_j^i \right\|_a^2 + \alpha^{-2} \left\| \mathbf{w} - T_i \bm{\phi}_j^i \right\|_a^2 }{2} + \frac{ \left\| \pi\left( T_i \bm{\phi}_j^i - T_{i,l} \bm{\phi}_j^i \right) \right\|_s^2 + \alpha^{-2} \left\| \pi\left( \mathbf{w} - T_i \bm{\phi}_j^i \right) \right\|_s^2 }{2}.
\end{aligned}
\]
That is,
\[
\begin{aligned}
& \left\| T_i \bm{\phi}_j^i - T_{i,l} \bm{\phi}_j^i \right\|_a^2 + \left\| \pi\left( T_i \bm{\phi}_j^i - T_{i,l} \bm{\phi}_j^i \right) \right\|_s^2  \leq \alpha^{-2} \left[ \left\| \mathbf{w} - T_i \bm{\phi}_j^i \right\|_a^2 + \left\| \pi\left( \mathbf{w} - T_i \bm{\phi}_j^i \right) \right\|_s^2 \right] \quad \forall \mathbf{w} \in \bm{V}(K_{i,l}).
\end{aligned}
\]
Letting $\mathbf{w} = \chi_i^{l,l-1} \left( T_i \bm{\phi}_j^i \right)$ in above inequality, we have
\begin{equation}
\label{leading inequality}
\begin{aligned}
\left\| T_i \bm{\phi}_j^i - T_{i,l} \bm{\phi}_j^i \right\|_a^2 + \left\| \pi\left( T_i \bm{\phi}_j^i - T_{i,l} \bm{\phi}_j^i \right) \right\|_s^2  \leq \alpha^{-2} \left[ \left\| \left( \chi_i^{l,l-1} - 1 \right) T_i \bm{\phi}_j^i \right\|_a^2 + \right.  \left. \left\| \pi\big( ( \chi_i^{l,l-1} - 1 ) T_i \bm{\phi}_j^i \big) \right\|_s^2 \right]. 
\end{aligned}
\end{equation}
Next we will estimate the two terms on the right-hand side of (\ref{leading inequality}). We divide it into 4 steps.

\textbf{Step 1:} 
We will estimate $\left\|\left(\chi_i^{l, l-1}-1\right) T_i \bm{\phi}_j^i\right\|_a^2$ in (\ref{leading inequality}). By the definition of the norm $\|\cdot\|_a$ and the fact that $\operatorname{supp}\left(1-\chi_i^{l, l-1}\right) \subset \Omega \setminus K_{i, l-1}$, we have
\begin{equation}
\label{estimate_for_first_term}
\begin{aligned}
\left\|\left(1-\chi_i^{l, l-1}\right) T_i \bm{\phi}_j^i\right\|_a^2 = & \int_{\Omega \setminus K_{i, l-1}} \mu_r^{-1} \left| \operatorname{curl} \big( (1-\chi_i^{l, l-1}) T_i \bm{\phi}_j^i \big) \right|^2 dx  + k^2 \int_{\Omega \setminus K_{i, l-1}} \left| \left(1-\chi_i^{l, l-1}\right) T_i \bm{\phi}_j^i \right|^2 dx \\
& + k \int_{\partial \Omega} \left| \left( \left(1-\chi_i^{l, l-1}\right) T_i \bm{\phi}_j^i \right)_T \right|^2 ds \\
\leq & \int_{\Omega \setminus K_{i, l-1}} \mu_r^{-1} \left| \nabla\left(1-\chi_i^{l, l-1}\right) \times T_i \bm{\phi}_j^i + \left(1-\chi_i^{l, l-1}\right) \operatorname{curl} (T_i \bm{\phi}_j^i) \right|^2 dx \\
& + k^2 \int_{\Omega \setminus K_{i, l-1}} \left| T_i \bm{\phi}_j^i \right|^2 dx + k \int_{\partial \Omega}  \left| \left( T_i \bm{\phi}_j^i \right)_T \right|^2 ds \\
\leq & \int_{\Omega \setminus K_{i, l-1}} \mu_r^{-1} \cdot 2 \left[ \left| \left( \nabla\chi_i^{l, l-1} \right) \times T_i \bm{\phi}_j^i \right|^2 + \left| \left(1-\chi_i^{l, l-1}\right) \operatorname{curl} (T_i \bm{\phi}_j^i) \right|^2 \right] dx \\
& + k^2 \int_{\Omega \setminus K_{i, l-1}} \left| T_i \bm{\phi}_j^i \right|^2 dx + k \int_{\partial \Omega} \left| \left( T_i \bm{\phi}_j^i \right)_T \right|^2 ds \\
\leq & \int_{\Omega \setminus K_{i, l-1}} \mu_r^{-1} \cdot 2 \left[ 2 \left| \nabla\chi_i^{l, l-1} \right|^2 \cdot \left| T_i \bm{\phi}_j^i \right|^2 + \left| \operatorname{curl} (T_i \bm{\phi}_j^i) \right|^2 \right] dx \\
& + k^2 \int_{\Omega \setminus K_{i, l-1}} \left| T_i \bm{\phi}_j^i \right|^2 dx + k \int_{\partial \Omega} \left| \left( T_i \bm{\phi}_j^i \right)_T \right|^2 ds \\
\leq & C_* \left( \left\| T_i \bm{\phi}_j^i \right\|_{s(\Omega \setminus K_{i, l-1})}^2 + \left\| T_i \bm{\phi}_j^i \right\|_{a(\Omega \setminus K_{i, l-1})}^2 \right)
\end{aligned}
\end{equation}
since $\left|\nabla\chi_i^{l, l-1}\right |^2 \sim O(H^{-2})$, where $C_*>0$ is independent of $H,k,\mu_r,\Lambda$. 
Note that for each $K_i$ ($1 \leq i \leq N$), we have
\begin{equation}
\label{T_i s_i inequality}
\begin{aligned}
\left\| T_i \bm{\phi}_j^i \right\|_{s_i}^2 & = \left\| (I-\pi) T_i \bm{\phi}_j^i + \pi\left( T_i \bm{\phi}_j^i \right) \right\|_{s_i}^2 = \left\| (I-\pi) T_i \bm{\phi}_j^i \right\|_{s_i}^2 + \left\| \pi\left( T_i \bm{\phi}_j^i \right) \right\|_{s_i}^2 \\
& \leq \frac{1}{\Lambda} \left\| T_i \bm{\phi}_j^i \right\|_{a_i}^2 + \left\| \pi\left( T_i \bm{\phi}_j^i \right) \right\|_{s_i}^2.
\end{aligned}
\end{equation}
Summing (\ref{T_i s_i inequality}) over all $K_i \subset \Omega \setminus K_{i, l-1}$ and combining (\ref{estimate_for_first_term}),
\begin{equation}
\label{step_1_estimate}
\left\| \left(1-\chi_i^{l, l-1}\right) T_i \bm{\phi}_j^i \right\|_a^2 \leq C_* \left(1+\frac{1}{\Lambda}\right) \left[ \left\| T_i \bm{\phi}_j^i \right\|_{a(\Omega \setminus K_{i, l-1})}^2 + \left\| \pi\left( T_i \bm{\phi}_j^i \right) \right\|_{s(\Omega \setminus K_{i, l-1})}^2 \right].
\end{equation}

\textbf{Step 2:} We will estimate the second term on the right-hand side of (\ref{leading inequality}). By using (\ref{T_i s_i inequality}), we have
\begin{equation}
\label{second term in the leading inequality}
\begin{aligned}
\left\| \pi\big( (1 - \chi_i^{l,l-1}) T_i \bm{\phi}_j^i \big) \right\|_s^2 & \leq \left\| \left(1 - \chi_i^{l,l-1}\right) T_i \bm{\phi}_j^i \right\|_s^2 \leq \left\| T_i \bm{\phi}_j^i \right\|_{s(\Omega \setminus K_{i,l-1})}^2 \\
& \leq \frac{1}{\Lambda} \left\| T_i \bm{\phi}_j^i \right\|_{a(\Omega \setminus K_{i,l-1})}^2 + \left\| \pi\left( T_i \bm{\phi}_j^i \right) \right\|_{s(\Omega \setminus K_{i,l-1})}^2.
\end{aligned}
\end{equation}
By combining (\ref{leading inequality}), (\ref{step_1_estimate}), and (\ref{second term in the leading inequality}), we have
\begin{equation*}
\begin{aligned}
 \left\| T_i \bm{\phi}_j^i - T_{i,l} \bm{\phi}_j^i \right\|_a^2 + \left\| \pi\left( T_i \bm{\phi}_j^i - T_{i,l} \bm{\phi}_j^i \right) \right\|_s^2 
\leq  \frac{C_*}{\alpha^2}\left(1 + \frac{1}{\Lambda}\right) \left[ \left\| T_i \bm{\phi}_j^i \right\|_{a(\Omega \setminus K_{i,l-1})}^2 + \left\| \pi\left( T_i \bm{\phi}_j^i \right) \right\|_{s(\Omega \setminus K_{i,l-1})}^2 \right].
\end{aligned}
\end{equation*}

\textbf{Step 3:} In this step, we estimate
$\left\| T_i \bm{\phi}_j^i \right\|_{a(\Omega \setminus K_{i,l-1})}^2 + \left\| \pi\left( T_i \bm{\phi}_j^i \right) \right\|_{s(\Omega \setminus K_{i,l-1})}^2.$
Let $\mathbf{v} = \left(1 - \chi_i^{l-1,l-2}\right) T_i \bm{\phi}_j^i$ in (\ref{variational form for global ms}), we have
\begin{equation}
\label{specific test function}
\begin{aligned}
& B\left( T_i \bm{\phi}_j^i, \left(1 - \chi_i^{l-1,l-2}\right) T_i \bm{\phi}_j^i \right) + s( \pi( T_i \bm{\phi}_j^i ), \pi( (1 - \chi_i^{l-1,l-2}) T_i \bm{\phi}_j^i ) ) \\
= & s( \bm{\phi}_j^i, \pi( (1 - \chi_i^{l-1,l-2}) T_i \bm{\phi}_j^i ) ) = 0
\end{aligned}
\end{equation}
since $\operatorname{supp}\left(1 - \chi_i^{l-1,l-2}\right) \subset \Omega \setminus K_{i,l-2}$ and $\operatorname{supp}\left(\bm{\phi}_j^i\right) \subset K_i$.

Notice that by Lemma \ref{covercity} and (\ref{specific test function}),
\[
\begin{aligned}
&\quad \left\| T_i \bm{\phi}_j^i \right\|_{a(\Omega \setminus K_{i,l-1})}^2 + \left\| \pi\left( T_i \bm{\phi}_j^i \right) \right\|_{s(\Omega \setminus K_{i,l-1})}^2  \leq \left\| \left(1 - \chi_i^{l-1,l-2}\right) T_i \bm{\phi}_j^i \right\|_a^2 + \left\| \pi\left( \left(1 - \chi_i^{l-1,l-2}\right) T_i \bm{\phi}_j^i \right) \right\|_s^2 \\
& \leq \alpha^{-1} \left| B\big((1 - \chi_i^{l-1,l-2}) T_i \bm{\phi}_j^i, (1 - \chi_i^{l-1,l-2}) T_i \bm{\phi}_j^i \big) \right. + \left. s\big( \pi((1 - \chi_i^{l-1,l-2}) T_i \bm{\phi}_j^i ), \pi( (1 - \chi_i^{l-1,l-2}) T_i \bm{\phi}_j^i ) \big) \right| \\
& \leq \alpha^{-1} \left| B\big( \chi_i^{l-1,l-2} T_i \bm{\phi}_j^i, (1 - \chi_i^{l-1,l-2}) T_i \bm{\phi}_j^i \big) \right. + \left. s\big( \pi( \chi_i^{l-1,l-2} T_i \bm{\phi}_j^i ), \pi( (1 - \chi_i^{l-1,l-2}) T_i \bm{\phi}_j^i)\big) \right|.
\end{aligned}
\]
Since $\operatorname{supp}\left(\chi_i^{l-1,l-2}\right) \subset K_{i,l-1}$ and $\operatorname{supp}\left(1-\chi_i^{l-1,l-2}\right) \subset \Omega \setminus K_{i,l-2}$, we know the two terms on the right-hand side of above inequality are nonzero in $K_{i,l-1} \setminus K_{i,l-2}$. 
Then combining the boundedness of $B(\cdot,\cdot)$ and (\ref{T_i s_i inequality}), we obtain that
\begin{equation*}
\begin{aligned}
&\quad \left\| T_i \bm{\phi}_j^i \right\|_{a(\Omega \setminus K_{i,l-1})}^2 + \left\| \pi\left( T_i \bm{\phi}_j^i \right) \right\|_{s(\Omega \setminus K_{i,l-1})}^2 \leq \alpha^{-1} \left[ \left\| T_i \bm{\phi}_j^i \right\|_{a(K_{i,l-1} \setminus K_{i,l-2})}^2 + \left\| T_i \bm{\phi}_j^i \right\|_{s(K_{i,l-1} \setminus K_{i,l-2})}^2 \right] \\
& \leq \alpha^{-1} (\frac{1}{\Lambda}+1)\left[ \left\| T_i \bm{\phi}_j^i \right\|_{a(K_{i,l-1} \setminus K_{i,l-2})}^2 + \left\| \pi\left( T_i \bm{\phi}_j^i \right) \right\|_{s(K_{i,l-1} \setminus K_{i,l-2})}^2 \right]. 
\end{aligned}
\end{equation*}

\textbf{Step 4:}
In this step, we will show that $\left\| T_i \bm{\phi}_j^i \right\|_{a(\Omega \setminus K_{i,l-1})}^2 + \left\| \pi\left( T_i \bm{\phi}_j^i \right) \right\|_{s(\Omega \setminus K_{i,l-1})}^2$ can be estimated by $\left\| T_i \bm{\phi}_j^i \right\|_{a(\Omega \setminus K_{i,l-2})}^2 + \left\| \pi\left( T_i \bm{\phi}_j^i \right) \right\|_{s(\Omega \setminus K_{i,l-2})}^2$. Based on Step 3, we have
\[
\begin{aligned}
& \quad \left\| T_i \bm{\phi}_j^i \right\|_{a(\Omega \setminus K_{i,l-2})}^2 + \left\| \pi\left( T_i \bm{\phi}_j^i \right) \right\|_{s(\Omega \setminus K_{i,l-2})}^2 \\
& = \left\| T_i \bm{\phi}_j^i \right\|_{a(\Omega \setminus K_{i,l-1})}^2 + \left\| \pi\left( T_i \bm{\phi}_j^i \right) \right\|_{s(\Omega \setminus K_{i,l-1})}^2 + \left\| T_i \bm{\phi}_j^i \right\|_{a(K_{i,l-1} \setminus K_{i,l-2})}^2 + \left\| \pi\left( T_i \bm{\phi}_j^i \right) \right\|_{s(K_{i,l-1} \setminus K_{i,l-2})}^2 \\
& \geq \left( 1 + \alpha \left( 1 + \frac{1}{\Lambda} \right)^{-1} \right) \left[ \left\| T_i \bm{\phi}_j^i \right\|_{a(\Omega \setminus K_{i,l-1})}^2 + \left\| \pi\left( T_i \bm{\phi}_j^i \right) \right\|_{s(\Omega \setminus K_{i,l-1})}^2 \right].
\end{aligned}
\]
Utilizing the above inequality recursively, we have
\[
\begin{aligned}
\left\| T_i \bm{\phi}_j^i \right\|_{a(\Omega \setminus K_{i,l-1})}^2 + \left\| \pi\left( T_i \bm{\phi}_j^i \right) \right\|^2_{s(\Omega \setminus K_{i,l-1})} \leq \left( 1 + \alpha \left( 1 + \frac{1}{\Lambda} \right)^{-1} \right)^{1-l} \left[ \left\| T_i \bm{\phi}_j^i \right\|_a^2 + \left\| \pi\left( T_i \bm{\phi}_j^i \right) \right\|_s^2 \right].
\end{aligned}
\]

Combining all steps above, we get
\[
\begin{aligned}
\left\| T_i \bm{\phi}_j^i - T_{i,l} \bm{\phi}_j^i \right\|_a^2 + \left\| \pi\left( T_i \bm{\phi}_j^i - T_{i,l} \bm{\phi}_j^i \right) \right\|_s^2 \leq \frac{C_*}{\alpha^2} \left( 1 + \frac{1}{\Lambda} \right) \left( 1 + \alpha \left( 1 + \frac{1}{\Lambda} \right)^{-1} \right)^{1-l} \left[ \left\| T_i \bm{\phi}_j^i \right\|_a^2 + \left\| \pi T_i \bm{\phi}_j^i \right\|_s^2 \right],
\end{aligned}
\]
where $C_*>0$ is independent of $H,k,\mu_r,\Lambda$.
Denote $0<\theta=\frac{(1+1/\Lambda)}{(1+1/\Lambda)+\alpha}<1$, then we obtain the desired result.
\end{proof}
Next we consider the inf-sup stability for the multiscale solution.

\begin{theorem}
\label{ms stability}
Under Assumption \ref{resolution condition}, the bilinear form $B$ satisfies the following inf-sup condition: there exists $\gamma(k, \mu_{\mathup{min}}) > 0$ such that
\[
\inf_{\mathbf{u}_H \in \bm{V}_{\mathup{ms}} \setminus \{\mathbf{0}\}} \sup_{\mathbf{u}_H^* \in \bm{V}_{\mathup{ms}}^* \setminus \{\mathbf{0}\}} \frac{\left| B(\mathbf{u}_H, \mathbf{u}_H^*) \right|}{\| \mathbf{u}_H \|_{k,\mathrm{imp}} \cdot \| \mathbf{u}_H^* \|_{k,\mathrm{imp}}} \geq \gamma(k, \mu_{\mathup{min}}) > 0,
\]
where $\gamma(k, \mu_{\mathup{min}})=\frac{1}{8C_k\max\{\sqrt{k},k\}} \left( 1 + \frac{1}{\alpha}\max \left\{\mu_{\min}^{-1}, 1 \right\} \right)^{-1}$.
\end{theorem}
\begin{proof}
For any $\mathbf{u}_H \in \bm{V}_{\mathup{ms}}$, we can find $\bm{\psi} \in \bm{V}_{\mathup{aux}}$ such that $\mathbf{u}_H = T_l \bm{\psi}$. We choose $T \bm{\psi} \in \bm{V}_{\mathup{glo}}$. By Theorem \ref{global well-posedness}, there exists $T^* \bm{\phi} \in \bm{V}_{\mathup{glo}}^*$ for some $\bm{\phi} \in \bm{V}_{\mathrm{\mathup{aux}}}$ such that
\[
\left| B(T \bm{\psi}, T^* \bm{\phi}) \right| \geq \beta(k, \mu_{\mathup{min}}) \| T \bm{\psi} \|_{k,\mathrm{imp}} \cdot \| T^* \bm{\phi} \|_{k,\mathrm{imp}}.
\]
Let $\mathbf{u}_H^* = T_l^* \bm{\phi}$, then we have
\begin{equation}
\label{leading control inequality}
\begin{aligned}
\left| B(\mathbf{u}_H, \mathbf{u}_H^*) \right| & = \left| B(T_l \bm{\psi}, T_l^* \bm{\phi}) \right| = \left| B(T \bm{\psi}, T^* \bm{\phi}) + B(T_l \bm{\psi} - T \bm{\psi}, T^* \bm{\phi}) \right. + \left. B(T_l \bm{\psi}, -T^* \bm{\phi} + T_l^* \bm{\phi}) \right| \\
& \geq \beta(k, \mu_{\mathup{min}}) \| T \bm{\psi} \|_{k,\mathrm{imp}} \cdot \| T^* \bm{\phi} \|_{k,\mathrm{imp}} - c_1 \| T_l \bm{\psi} - T \bm{\psi} \|_{k,\mathrm{imp}} \cdot \| T^* \bm{\phi} \|_{k,\mathrm{imp}} \\
& \quad - c_2 \| T_l \bm{\psi} \|_{k,\mathrm{imp}} \cdot \| T_l^* \bm{\phi} - T^* \bm{\phi} \|_{k,\mathrm{imp}}\\
& \geq \beta(k, \mu_{\mathup{min}}) \| T \bm{\psi} \|_{k,\mathrm{imp}} \cdot \| T^* \bm{\phi} \|_{k,\mathrm{imp}} - c_1 \| T_l \bm{\psi} - T \bm{\psi} \|_{k,\mathrm{imp}} \cdot \| T^* \bm{\phi} \|_{k,\mathrm{imp}} \\
& \quad - c_2 \| T_l \bm{\psi} - T \bm{\psi} \|_{k,\mathrm{imp}} \cdot \| T_l^* \bm{\phi} - T^* \bm{\phi} \|_{k,\mathrm{imp}} - c_2 \| T \bm{\psi} \|_{k,\mathrm{imp}} \cdot \| T_l^* \bm{\phi} - T^* \bm{\phi} \|_{k,\mathrm{imp}},
\end{aligned}
\end{equation}
where $c_1$, $c_2$ depend on $\mu_\mathup{min}$ by the boundedness of $B$ (see (\ref{boundedness})). By the decay property of the multiscale basis functions (i.e., Theorem \ref{decay}), Assumption \ref{regularity for mesh}, Lemma \ref{property of pi} and the equivalence of $\|\cdot\|_{k,\mathrm{imp}}$ and $\|\cdot\|_a$, we have
\begin{equation}
\label{controlled by T term}
\begin{aligned}
\| T_l \bm{\psi} - T \bm{\psi} \|_{k,\mathrm{imp}} \leq CC_{\mathup{ol}}(l+1)^d (1 + \frac{1}{\Lambda})\frac{\sqrt{C_*}}{\alpha} \theta^{(l-1)/2} \| T \bm{\psi} \|_{k,\mathrm{imp}}.
\end{aligned}
\end{equation}
Denote $D_1 = C(l+1)^d (1 + \frac{1}{\Lambda})\frac{\sqrt{C_*}}{\alpha} \theta^{(l-1)/2}$. Similarly, for the adjoint operator $T^*$, $T_l^*$, we have
\begin{equation}
\label{controlled by T term adjoint}
\| T_l^* \bm{\phi} - T^* \bm{\phi} \|_{k,\mathrm{imp}} \leq D_2 \| T^* \bm{\phi} \|_{k,\mathrm{imp}},
\end{equation}
where $D_2$ also contains the decay term like $\theta^{(l-1)/2}$ in $D_1$. Note that by selecting proper $l$ (i.e., using the exponential decay property) in (\ref{controlled by T term})-(\ref{controlled by T term adjoint}), we can let $D_1$, $D_2$ small enough to make the last three terms on the right-hand side of (\ref{leading control inequality}) controlled by $\frac{\beta(k, \mu_{\mathup{min}}) }{2} \| T \bm{\psi} \|_{k,\mathrm{imp}}\cdot \| T^* \bm{\phi} \|_{k,\mathrm{imp}}$. Thus, we have
\[
\left| B(\mathbf{u}_H, \mathbf{u}_H^*) \right| \geq \frac{\beta(k, \mu_{\mathup{min}}) }{2} \| T \bm{\psi} \|_{k,\mathrm{imp}}\cdot \| T^* \bm{\phi} \|_{k,\mathrm{imp}}.
\]
For $\| T \bm{\psi} \|_{k,\mathrm{imp}}$, $\| T^* \bm{\phi} \|_{k,\mathrm{imp}}$, by the triangle inequality, we have
\[
\| T \bm{\psi} \|_{k,\mathrm{imp}} \geq \| T_l \bm{\psi} \|_{k,\mathrm{imp}} - \| T_l \bm{\psi} - T \bm{\psi} \|_{k,\mathrm{imp}},
\quad
\left\| T^* \bm{\phi} \right\|_{k,\mathrm{imp}} \geq \left\| T_l^* \bm{\phi} \right\|_{k,\mathrm{imp}} - \left\| T_l^* \bm{\phi} - T^* \bm{\phi} \right\|_{k,\mathrm{imp}}.
\]
Therefore, we obtain that
\[
\begin{aligned}
\left| B(\mathbf{u}_H, \mathbf{u}_H^*) \right| & \geq \frac{\beta(k, \mu_{\mathup{min}})}{4} \| T \bm{\psi} \|_{k,\mathrm{imp}} \cdot \| T^* \bm{\phi} \|_{k,\mathrm{imp}} + \frac{\beta(k, \mu_{\mathup{min}})}{4} \| T_l \bm{\psi} \|_{k,\mathrm{imp}} \cdot \| T_l^* \bm{\phi} \|_{k,\mathrm{imp}} \\
& \quad - \frac{\beta(k, \mu_{\mathup{min}})}{4} ( \| T_l \bm{\psi} \|_{k,\mathrm{imp}} \cdot \| T_l^* \bm{\phi} - T^* \bm{\phi} \|_{k,\mathrm{imp}} + \| T_l^* \bm{\phi} \|_{k,\mathrm{imp}} \cdot \| T_l \bm{\psi} - T \bm{\psi} \|_{k,\mathrm{imp}}).
\end{aligned}
\]
Similar to the previous analysis, the last two terms can be controlled by the first term of the above inequality. More precisely,
\[
\begin{aligned}
&\quad - ( \| T_l \bm{\psi} \|_{k,\mathrm{imp}} \cdot \| T_l^* \bm{\phi} - T^* \bm{\phi} \|_{k,\mathrm{imp}} + \| T_l^* \bm{\phi} \|_{k,\mathrm{imp}} \cdot \| T_l \bm{\psi} - T \bm{\psi} \|_{k,\mathrm{imp}} ) \\
& \geq - ( \| T_l \bm{\psi} - T \bm{\psi} \|_{k,\mathrm{imp}} \cdot \| T_l^* \bm{\phi} - T^* \bm{\phi} \|_{k,\mathrm{imp}} + \| T \bm{\psi} \|_{k,\mathrm{imp}} \cdot \| T_l^* \bm{\phi} - T^* \bm{\phi} \|_{k,\mathrm{imp}}  \\
& \quad + \| T_l^* \bm{\phi} - T^* \bm{\phi} \|_{k,\mathrm{imp}} \cdot \| T_l \bm{\psi} - T \bm{\psi} \|_{k,\mathrm{imp}} + \| T^* \bm{\phi} \|_{k,\mathrm{imp}} \cdot \| T_l \bm{\psi} - T \bm{\psi} \|_{k,\mathrm{imp}} ),
\end{aligned}
\]
which can be controlled by $\| T \bm{\psi} \|_{k,\mathrm{imp}} \cdot \| T^* \bm{\phi} \|_{k,\mathrm{imp}}$ for proper $l$.
Thus, we conclude that there exists $\gamma(k, \mu_{\mathup{min}})=\frac{\beta(k, \mu_{\mathup{min}})}{4}=\frac{1}{8C_k\max\{\sqrt{k},k\}} \left( 1 + \frac{1}{\alpha}\max \left\{\mu_{\min}^{-1}, 1 \right\} \right)^{-1}$ such that
\[
\begin{aligned}
\left| B(\mathbf{u}_H, \mathbf{u}_H^*) \right| \geq \gamma(k, \mu_{\mathup{min}}) \| T_l \bm{\psi} \|_{k,\mathrm{imp}} \cdot \| T_l^* \bm{\phi} \|_{k,\mathrm{imp}} = \gamma(k, \mu_{\mathup{min}}) \| \mathbf{u}_H \|_{k,\mathrm{imp}} \cdot \| \mathbf{u}_H^* \|_{k,\mathrm{imp}}.
\end{aligned}
\]
\end{proof}
Finally, we state and prove the convergence theorem.  We first give an assumption.
\begin{assumption}
\label{regularity for mesh}
There exists a positive constant $C_\text{ol}$ such that for all $K_j\in\mathcal{T}_H$ and $m>0$, 
\[
\#\{K\in \mathcal{T}_H|\, K\subset K_{j,m}\}\leq C_\text{ol}m^d.
\]
\end{assumption}
\begin{theorem}
\label{local convergence}
Let $\mathbf{u}$ be the exact solution of (\ref{weak form}) and $\mathbf{u}_{\mathup{ms}}$ be the multiscale solution of (\ref{multiscale discretization}). Then we have the following error estimate
\begin{equation}
\label{final error}
\begin{aligned}
\|\mathbf{u} - \mathbf{u}_{\mathrm{ms}}\|_a \leq & 
 C(k,\mu) \Bigg( \frac{1}{\alpha\sqrt{\Lambda}} \|\mathbf{f}\|_{s^{-1}} 
\\
&+ \sqrt{C_{\mathrm{ol}}(l+1)^d} \cdot \frac{\sqrt{C_*}}{\alpha} \sqrt{1 + \frac{1}{\Lambda}}  \cdot \theta^{(l-1)/2} \cdot (\sqrt{D} + 1) 
\cdot \big( \|\mathbf{u}_{\mathrm{glo}}\|_a + \|\pi(\mathbf{u}_{\mathrm{glo}})\|_s \big) \Bigg),
\end{aligned}
\end{equation}
where $C(k,\mu)=\sqrt{\max\left\{\frac{\mu_{\max}}{\mu_{\min}}, \mu_{\max}\right\}}\left( 1 + \frac{\max\{\mu_{\min}^{-1}, 1\}}{\gamma(k, \mu_{\min})} \right)$, $\gamma(k, \mu_{\min})=\frac{1}{8C_k\max\{\sqrt{k},k\}} \left( 1 + \frac{1}{\alpha}\max \left\{\mu_{\min}^{-1}, 1 \right\} \right)^{-1}$, $0<\theta=\frac{(1+1/\Lambda)}{(1+1/\Lambda)+\alpha}<1$, and $\alpha,C_*,C_{\mathup{ol}},D$ are independent of $H,\Lambda,\mu_r,k$.
\end{theorem}
\begin{remark}
\label{finalremark}
Observing Theorem \ref{local convergence}, we give more details about the error. It's clear that $\| \mathbf{f} \|_{s^{-1}}=O(\sqrt{\mu_\mathup{max}}H)$. In addition, $ \left\| \pi\left(\mathbf{u}_{\mathup{glo}}\right) \right\|_s=\left\| \pi\mathbf{u}\right\|_s=O(H^{-1}\mu_\mathup{min}^{-1/2})$. $\left\| \mathbf{u}_{\mathup{glo}} \right\|_a\leq \left\| \mathbf{u} - \mathbf{u}_{\mathup{glo}} \right\|_a + \left\| \mathbf{u} \right\|_a\leq O(\sqrt{\mu_\mathup{max}}H)+O(C_k)$, where $C_k$ can be bounded by $\mu_{\max}$ under the high-contrast setting.
{Then by letting $C(k,\mu)\alpha^{-1}\Lambda^{-1/2}\leq C$ and choosing the oversampling size $l$ such that
\[
(l+1)^{d/2} \cdot \theta^{(l-1)/2}\cdot \sqrt{\frac{\mu_{\text{max}} }{\mu_\mathup{min}} } \sim O(H),
\]
we will have an $O(H)$ convergence. Furthermore, according to the error bound in Theorem \ref{local convergence} (see (\ref{final error})), the approximation error grows with increasing wave number $k$.}
\end{remark}
\begin{proof}
In terms of Lemma \ref{covercity} and (\ref{variational form for global ms}), we have
\[
\left\| T_i \bm{\phi}_j^i \right\|_a^2 + \left\| \pi\left( T_i \bm{\phi}_j^i \right) \right\|_s^2 \leq \alpha^{-1}\big(B\left( T_i \bm{\phi}_j^i, T_i \bm{\phi}_j^i \right) + s\left( \pi\left( T_i \bm{\phi}_j^i \right), \pi\left( T_i \bm{\phi}_j^i \right) \right)\big) = \alpha^{-1}s\left( \bm{\phi}_j^i, \pi\left( T_i \bm{\phi}_j^i \right) \right).
\]
Then we obtain that $\left\| \pi\left( T_i \bm{\phi}_j^i \right) \right\|_s^2 \leq \alpha^{-1}\left\| \bm{\phi}_j^i \right\|_s \cdot \left\| \pi\left( T_i \bm{\phi}_j^i \right) \right\|_s$, that is,
\begin{equation*}
\left\| \pi\left( T_i \bm{\phi}_j^i \right) \right\|_s \leq \alpha^{-1}\left\| \bm{\phi}_j^i \right\|_s = \alpha^{-1}
\end{equation*}
and
\[
\left\| T_i \bm{\phi}_j^i \right\|_a^2\leq \alpha^{-1}\left\| \bm{\phi}_j^i \right\|_s\left\| \pi\left( T_i \bm{\phi}_j^i \right) \right\|_s\leq \alpha^{-2}.
\]
By combining above two estimates and Theorem \ref{decay}, we obtain
\[
\begin{aligned}
 \left\| T_i \bm{\phi}_j^i - T_{i,l} \bm{\phi}_j^i \right\|_a^2 + \left\| \pi\left( T_i \bm{\phi}_j^i \right) - \pi\left( T_{i,l} \bm{\phi}_j^i \right) \right\|_s^2 \leq \frac{C_*}{\alpha^2} \left( 1 + \frac{1}{\Lambda} \right)\theta^{l-1} \cdot 2\alpha^{-2}.
\end{aligned}
\]
We write $\mathbf{u}_{\mathup{glo}} = \sum_{i=1}^N \sum_{j=1}^{l_i} c_j^i T_i \bm{\phi}_j^i$ and $\mathbf{w} = \sum_{i=1}^N \sum_{j=1}^{l_i} c_j^i T_{i,l} \bm{\phi}_j^i \in \bm{V}_{\mathup{ms}}$.
By (\ref{weak form}) and (\ref{variational form for global ms}), we clearly have
\begin{equation}
\label{u-u_ms}
B\left( \mathbf{u} - \mathbf{u}_{\mathup{ms}}, \mathbf{v} \right) = 0 \quad \forall \mathbf{v} \in \bm{V}_{\mathup{ms}}^*.
\end{equation}
Following Theorem \ref{ms stability} and (\ref{u-u_ms}), (\ref{boundedness}), we have
\[
\left\| \mathbf{w} - \mathbf{u}_{\mathup{ms}} \right\|_{k,\mathrm{imp}} \leq \frac{1}{\gamma(k, \mu_{\mathup{min}})} \sup_{\mathbf{v} \in \bm{V}_{\mathup{ms}}^* \setminus \{\mathbf{0}\}} \frac{\left| B(\mathbf{w} - \mathbf{u}_{\mathup{ms}}, \mathbf{v}) \right|}{\| \mathbf{v} \|_{k,\mathrm{imp}}}.
\]
\[
\begin{aligned}
& = \frac{1}{\gamma(k, \mu_{\mathup{min}})} \sup_{\mathbf{v} \in \bm{V}_{\mathup{ms}}^* \setminus \{\mathbf{0}\}} \frac{|B(\mathbf{w} - \mathbf{u}, \mathbf{v})|}{\|\mathbf{v}\|_{k,\mathrm{imp}}} \\
& \leq \frac{\max \left\{\mu_{\min}^{-1}, 1 \right\}}{\gamma(k, \mu_{\mathup{min}})} \|\mathbf{w} - \mathbf{u}\|_{k,\mathrm{imp}}.
\end{aligned}
\]

Thus,
\[
\begin{aligned}
\|\mathbf{u} - \mathbf{u}_{\mathup{ms}}\|_{k,\mathrm{imp}} & \leq \|\mathbf{u} - \mathbf{w}\|_{k,\mathrm{imp}} + \|\mathbf{w} - \mathbf{u}_{\mathup{ms}}\|_{k,\mathrm{imp}}  \leq \left( 1 + \frac{\max \left\{\mu_{\min}^{-1}, 1 \right\}}{\gamma(k, \mu_{\mathup{min}})} \right) \|\mathbf{w} - \mathbf{u}\|_{k,\mathrm{imp}} \\
& \leq \left( 1 + \frac{\max \left\{\mu_{\min}^{-1}, 1 \right\}}{\gamma(k, \mu_{\mathup{min}})} \right) \left( \|\mathbf{u} - \mathbf{u}_{\mathup{glo}}\|_{k,\mathrm{imp}} + \|\mathbf{u}_{\mathup{glo}} - \mathbf{w}\|_{k,\mathrm{imp}} \right).
\end{aligned}
\]
By the equivalence of norms $\|\cdot\|_a$ and $\|\cdot\|_{k,\mathrm{imp}}$ (\ref{equivalence for two norms}),
\begin{equation}
\label{u-ums first}
\begin{aligned}
\|\mathbf{u} - \mathbf{u}_{\mathup{ms}}\|_a & \leq \sqrt{\max\{\mu_{\mathup{min}}^{-1}, 1\}} \left( 1 + \frac{\max \left\{\mu_{\min}^{-1}, 1 \right\}}{\gamma(k, \mu_{\mathup{min}})} \right) \sqrt{\max\{\mu_{\mathup{max}}, 1\}} \cdot \left( \|\mathbf{u} - \mathbf{u}_{\mathup{glo}}\|_a + \|\mathbf{u}_{\mathup{glo}} - \mathbf{w}\|_a \right)\\
&\leq \sqrt{\max\{\frac{\mu_{\mathup{max}}}{\mu_{\mathup{min}}}, \mu_{\mathup{max}}\}} \left( 1 + \frac{\max \left\{\mu_{\min}^{-1}, 1 \right\}}{\gamma(k, \mu_{\mathup{min}})} \right)\cdot \left( \|\mathbf{u} - \mathbf{u}_{\mathup{glo}}\|_a + \|\mathbf{u}_{\mathup{glo}} - \mathbf{w}\|_a \right). 
\end{aligned}
\end{equation}
By utilizing Assumption \ref{regularity for mesh}, Theorem \ref{decay}, the fact that $s\left(\bm{\phi}_j^i, \bm{\phi}_l^i\right) = \delta_{jl}$ (applying them to the function $\sum_{i=1}^N \sum_{j=1}^{l_i} c_j^i \left( T_i \bm{\phi}_j^i - T_{i,l} \bm{\phi}_j^i \right)$) and denoting $\bm{\phi} = \sum_{i=1}^N \sum_{j=1}^{l_i} c_j^i \bm{\phi}_j^i \in \bm{V}_{\mathrm{\mathup{aux}}}$, we obtain
\[
\begin{aligned}
\|\mathbf{w} - \mathbf{u}_{\mathup{glo}}\|_a^2 & \leq C_{\mathup{ol}} (l+1)^d \sum_{i=1}^N \sum_{j=1}^{l_i} \left\| c_j^i \left( T_i \bm{\phi}_j^i - T_{i,l} \bm{\phi}_j^i \right) \right\|_a^2 \\
& \leq C_{\mathup{ol}} (l+1)^d \cdot \frac{C_*}{\alpha^2} \left( 1 + \frac{1}{\Lambda} \right)\theta^{l-1} \sum_{i=1}^N \sum_{j=1}^{l_i} \left\| c_j^i \bm{\phi}_j^i \right\|_s^2 \\
& = C_{\mathup{ol}} (l+1)^d \cdot \frac{C_*}{\alpha^2} \left( 1 + \frac{1}{\Lambda} \right)\theta^{l-1} \cdot s(\bm{\phi}, \bm{\phi}).
\end{aligned}
\]
By the definition of $\mathbf{u}_{\mathup{glo}}$, $\bm{\phi}$ and the variational form (\ref{variational form for global ms}), we know that
\begin{equation}
\label{u_glo phi relation}
B\left(\mathbf{u}_{\mathup{glo}}, \mathbf{v} \right) + s\left(\pi\left(\mathbf{u}_{\mathup{glo}}\right), \pi(\mathbf{v})\right) = s(\bm{\phi}, \pi(\mathbf{v})) \quad \forall \mathbf{v} \in \bm{V}. 
\end{equation}
For $\bm{\phi} \in \bm{V}_{\mathrm{\mathup{aux}}}$, by Lemma \ref{aux_lemma}, there is $\bm{\xi} \in \bm{V}$ such that $\pi(\bm{\xi}) = \bm{\phi}$, $\|\bm{\xi}\|_a^2 \leq D \|\bm{\phi}\|_s^2$. Letting $\mathbf{v} = \bm{\xi}$ in (\ref{u_glo phi relation}), we have 
\[
B\left(\mathbf{u}_{\mathup{glo}}, \bm{\xi}\right) + s\left(\pi\left(\mathbf{u}_{\mathup{glo}}\right), \pi(\bm{\xi})\right) = s(\bm{\phi}, \pi(\bm{\xi})) = s(\bm{\phi}, \bm{\phi}).
\]
Then we obtain that
\[
\begin{aligned}
s(\bm{\phi}, \bm{\phi}) & = B\left(\mathbf{u}_{\mathup{glo}}, \bm{\xi}\right) + s\left(\pi\left(\mathbf{u}_{\mathup{glo}}\right), \pi(\bm{\xi})\right) \leq \left\| \mathbf{u}_{\mathup{glo}} \right\|_a \cdot \|\bm{\xi}\|_a + \left\| \pi\left(\mathbf{u}_{\mathup{glo}}\right) \right\|_s \cdot \|\bm{\phi}\|_s \\
& \leq \left( \sqrt{D} + 1 \right) \left( \left\| \mathbf{u}_{\mathup{glo}} \right\|_a + \left\| \pi\left(\mathbf{u}_{\mathup{glo}}\right) \right\|_s \right) \cdot \|\bm{\phi}\|_s.
\end{aligned}
\]
Therefore, we have
\begin{equation}
\label{u-ums first 2rd term}
\left\| \mathbf{w} - \mathbf{u}_{\mathup{glo}} \right\|_a \leq \sqrt{ C_{\mathup{ol}} (l+1)^d} \cdot \frac{\sqrt{C_*}}{\alpha} \sqrt{1 + \frac{1}{\Lambda}}\cdot\theta^{(l-1)/2}\cdot \left( \sqrt{D} + 1 \right) \cdot \left( \left\| \mathbf{u}_{\mathup{glo}} \right\|_a + \left\| \pi\left(\mathbf{u}_{\mathup{glo}}\right) \right\|_s \right).
\end{equation}
Combining (\ref{u-ums first}), (\ref{u-ums first 2rd term}) and Theorem \ref{global convergence}, we can get the desired results.
\end{proof}
\section{Numerical experiments}\label{sec:Numerical experiments}
In this section, we present some numerical examples in 2D and 3D to demonstrate the performance of the proposed methods. For the 2D experiments, we consider the numerical experiments on a unit square $\Omega=(0,1)^2$ with fine mesh $h=1/256$. For the coarse mesh sizes $H$, we consider coarse meshes to be $4\times 4$, $8 \times 8$, $16 \times 16$, and $32 \times 32$. For the 3D experiments, We conduct all numerical experiments on a unit cube $\Omega=(0,1)^3$, therefore the media term $\mu_r(x)$ is generated from $64\mathup{px}\times 64\mathup{px}\times 64\mathup{px}$ figures. For the coarse grid sizes $H$, we choose $H$ to be $1/4, 1/8 $ and $1/16$.  If $\bold{e}_h$ denotes the difference between the multiscale approximation $\bold{u}$ and the reference solution $\bold{u}_h$, we calculate the relative $a$-error and $L^2$-error defined by
\[
  \frac{\norm{\bold{e}_h}_{a(\Omega)}}{\norm{\bold{u}_h}_{a(\Omega)}} \text{  and  } \frac{\norm{\bold{e}_h}_{L^2(\Omega)}}{\norm{\bold{u}_h}_{L^2(\Omega)}},
\]
where $\bold{u}_h$ is the exact
solution (if available), or the reference solution calculated by the traditional FEM in the first-order Nédélec space on $\mathcal{T}_h$ with mesh size $ h$. All the numerical experiments were performed in Python libraries Numpy and SciPy using VSCode on a machine equipped with a 12th‑generation Intel Core i9‑12900 processor running at 2.40 GHz.
\subsection{Homogeneous structures in 3D}
We firstly consider the coefficient $\mu_r^{-1}=1$ with wave number $k=4$. The right-hand side $\bold{f} =(-k^2\sin(kx),0,0)^T$ and the impedance boundary conditions $\bold{g}$ are chosen in (\ref{eq:bc1}) such that the problem (\ref{model}) admits the exact solution $\bold{u} =\sin(kx)(1,1,1)^T$.
\begin{equation}\label{eq:bc1}
\bold{g}(x, y, z)=
\begin{cases}
(0,k\cos k+\mathrm{i}k\sin k,k\cos k+\mathrm{i}k\sin k)^T,&\text{on}\quad \{1\}\times(0,1)\times (0,1),\\
(-k\cos(kx)+\mathrm{i}k\sin(kx),0,-\mathrm{i}k\sin(kx))^T,&\text{on}\quad (0,1)\times\{1\}\times(0,1),\\
(-k\cos(kx)+\mathrm{i}k\sin(kx),\mathrm{i}k\sin(kx),0)^T,&\text{on}\quad(0,1)\times (0,1)\times \{1\},\\
(0,0,0)^T,&\text{on}\quad\{0\}\times(0,1)\times(0,1),\\
(k\cos(kx)+\mathrm{i}k\sin(kx),0,\mathrm{i}k\sin(kx))^T&\text{on}\quad(0,1)\times \{0\}\times(0,1),\\
(k\cos(kx)+\mathrm{i}k\sin(kx),\mathrm{i}k\sin(kx),0)^T&\text{on}\quad(0,1)\times(0,1) \times \{0\}.
\end{cases}
\end{equation}
For the setting of the proposed multiscale method, we fix $l_i = 4$, indicating
that we calculate the first four eigenfunctions and construct four multiscale
bases for each coarse element, while we vary the oversampling layers m from
1 to 4. For the relative error norms, we choose $\bold{u}_h$ to be the exact solution of this model, $\bold{u}$ is approximated by the CEM-GMsFEM method. We also show the relative error between the exact solution and the approximate solution obtained from the traditional FEM in the first-order Nédélec space. We refer to  \cref{tab:real-error} and \cref{fig:real-con} for the numerical results. In \cref{fig:real-con}, 
the convergence of the  FEM manifests a linear pattern w.r.t.
$H$ in the logarithmic scale, consistent with the theoretical expectation. For the convergence of the multiscale method, the convergence rate does not always exhibit a linear trend. This behavior is due to the relationship between the number of oversampling layers and the mesh size, as discussed in Remark \ref{finalremark}. Interestingly, when 
$m=4$, the convergence rates in both norms show good performance, and the resulting errors are significantly smaller than those of the FEM method. In \cref{tab:wave-error}, we vary the wave number to investigate its effect on the error. In order to avoid the pollution effect and to satisfy Assumption~\ref{resolution condition}, the errors in both the energy norm and the $L^2$
norm increase across each column as the wave number grows. This observation indicates that our results are strongly dependent on the wave number. As shown in Theorem~\ref{local convergence}, the right-hand side of the energy-norm estimate depends explicitly on 
$k$ through the parameter $C(k,\mu)$.
\begin{figure}[!ht]
    \centering
    \includegraphics[width=\linewidth]{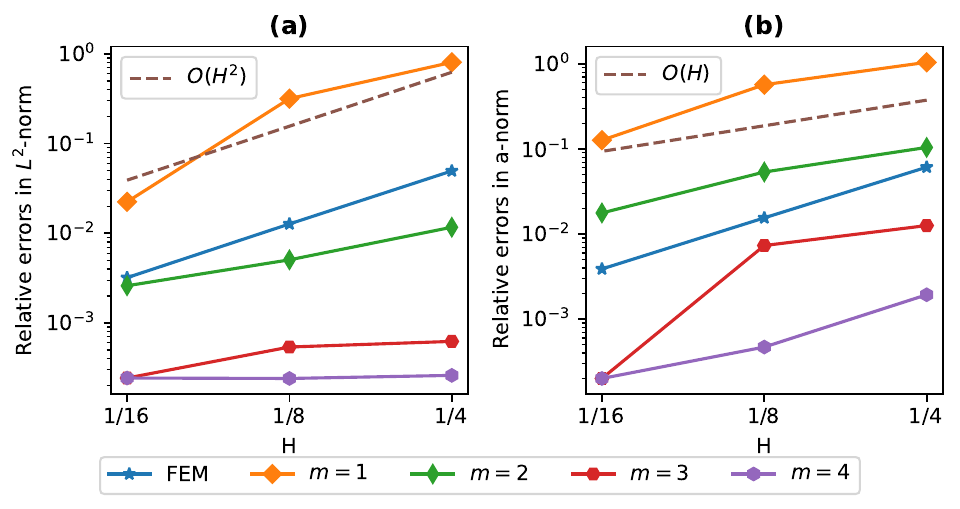}
    \caption{ Numerical results for the Homogeneous structures ($\mu_r=1)$. Subplots (a) and (b) show the relative errors of the proposed method with different numbers of oversampling layers $m$ and the  FEM w.r.t. the coarse mesh size $H$, but measured in different norms.}
    \label{fig:real-con}
\end{figure}
\begin{table}[!ht]
\caption{
The relative errors in the $a$-norm (in the columns labelled with $\norm{\cdot}_{a}$) and in the $L^2$ norm (in the columns labelled with $\norm{\cdot}_{L^2(\Omega)}$).
  }\label{tab:real1}
\centering
\resizebox{\textwidth}{!}{
\makegapedcells
\footnotesize{
\begin{tabular}{ c c c c c c c c c}
\toprule
\multirow{3}{*}{$H$} & \multicolumn{2}{c}{$FEM$}  & \multicolumn{2}{c}{$m=2$}    & \multicolumn{2}{c}{$m=3$}  & \multicolumn{2}{c}{$m=4$}\\
\cmidrule{2-9}
 & $\norm{\cdot}_{a(\Omega)}$ & $\norm{\cdot}_{L^2(\Omega)}$ & $\norm{\cdot}_{a(\Omega)}$ & $\norm{\cdot}_{L^2(\Omega)}$ & $\norm{\cdot}_{a(\Omega)}$ & $\norm{\cdot}_{L^2(\Omega)}$ & $\norm{\cdot}_{a(\Omega)}$ & $\norm{\cdot}_{L^2(\Omega)}$ \\
\midrule            
$\frac{1}{4}$       & \num{6.107e-02}            & \num{4.955e-02}               & \num{1.043e-01}            & \num{1.167e-02}              & \num{1.257e-02}            & \num{6.224e-04}              & \num{1.932e-03}            & \num{2.600e-04}              \\
$\frac{1}{8}$       & \num{ 1.549e-02}            & \num{1.270e-02}              & \num{5.359e-02}            & \num{5.055e-03}              & \num{7.323e-03}            & \num{5.387e-04}              & \num{4.690e-04}            & \num{2.403e-04}              \\
$\frac{1}{16}$       & \num{3.884e-03 }            & \num{ 3.198e-03}              & \num{1.769e-02}            & \num{2.597e-03}              & \num{2.003e-04}            & \num{2.429e-04}              & \num{2.003e-04}            & \num{2.429e-04}              \\  
\bottomrule
\end{tabular}
}
}
\label{tab:real-error}
\end{table}
\begin{table}[!ht]
\caption{The relative errors in the $a$-norm (in the columns labelled with $\norm{\cdot}_{a(\Omega)}$) and in the $L^2$ norm (in the columns labelled with $\norm{\cdot}_{L^2(\Omega)}$) with different wave number $k$.}
\centering
\resizebox{\textwidth}{!}{
\makegapedcells
\footnotesize{
\begin{tabular}{ c c c c c c c c c c}
\toprule
\multirow{3}{*}{$k$} & \multirow{3}{*}{$H$}  & \multicolumn{2}{c}{$FEM$}  & \multicolumn{2}{c}{$m=2$}    & \multicolumn{2}{c}{$m=3$}  & \multicolumn{2}{c}{$m=4$}\\
\cmidrule{3-10}
& & $\norm{\cdot}_{a(\Omega)}$ & $\norm{\cdot}_{L^2(\Omega)}$ & $\norm{\cdot}_{a(\Omega)}$ & $\norm{\cdot}_{L^2(\Omega)}$ & $\norm{\cdot}_{a(\Omega)}$ & $\norm{\cdot}_{L^2(\Omega)}$ & $\norm{\cdot}_{a(\Omega)}$ & $\norm{\cdot}_{L^2(\Omega)}$ \\
\midrule            
$4$  &$\frac{1}{4}$  & \num{6.107e-2}   & \num{4.955e-2}              & \num{5.559e-03}            & \num{1.715e-03}              & \num{2.429e-04}            & \num{2.003e-04}              & \num{2.492e-04}            & \num{2.003e-04}           \\
$8$    &$\frac{1}{8}$    & \num{9.219e-02}            & \num{9.807e-02}              & \num{2.366e-02}            & \num{3.875e-03}              & \num{5.559e-03}            & \num{1.715e-03}              & \num{1.796e-03}            & \num{1.636e-03}              \\
$16$   &$\frac{1}{16}$     & \num{1.909e-01}            & \num{1.929e-01}              & \num{3.250e-02}            & \num{1.691e-02}              & \num{1.347e-02}            & \num{1.288e-01}              & \num{1.279e-01}            & \num{1.267e-01}              \\      
$32$    &$\frac{1}{32}$    & \num{3.879e-01}            & \num{3.863e-01}              & \num{1.028e-01}            & \num{1.027e-01}              & \num{1.021e-01}            & \num{1.019e-01}              & \num{1.021e-01}            & \num{1.019e-01}              \\
\bottomrule
\end{tabular}
}
}
\label{tab:wave-error}
\end{table}
\subsection{High-contrast photonic band structures in 3D}  In this example, we test the robustness of our methods for domains with high-contrast photonic band structures, which are shown in \cref{{fig:highcon}}. In particular, we denote the coefficient \(\mu_r^{-1}\) and illustrated in \cref{{fig:highcon}}-(a), as corresponding to \textbf{Model 1}, and the photonic crystal structure with holes, shown in \cref{{fig:highcon}}-(b), as \textbf{Model 2}. Both models are often compared in literature \cite{Ma2025, Verfuerth2019} because they exhibit similar photonic bandgaps for certain polarizations (e.g., TM modes in rods-in-air vs. TE modes in holes-in-slab \cite{Johnson1999}), but their effective medium descriptions differ. For the source and boundary terms, we choose $\bold{f}(x,y,z)=(1,1,1)$ and suitable $\bold{g}(x,y,z)$ for our test. Due to the absence of exact solutions for these two models, we compute the relative error norms by taking $\bold{u}_h$ to be approximated by the CEM-GMsFEM method, while the reference solution is obtained using the standard FEM in the first-order the first-order Nédélec space.
\begin{figure}[!ht]    
\centering
    \includegraphics[width=\linewidth]{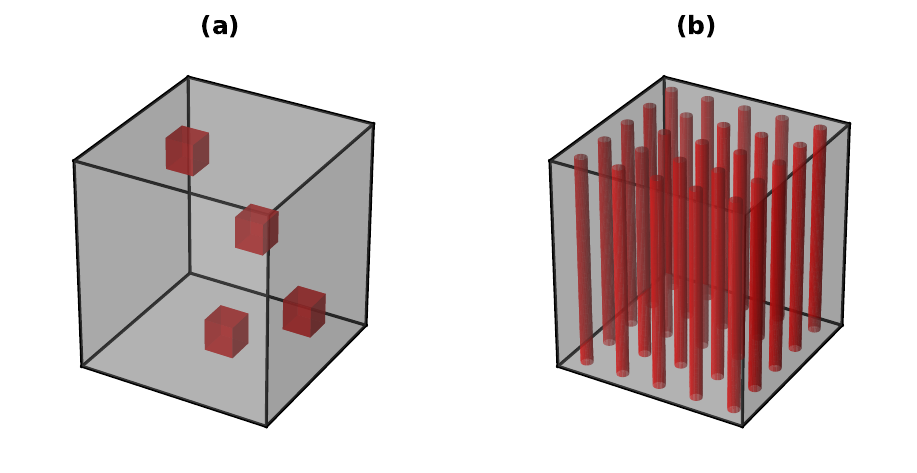}
\caption{Computational domains:(a) \textbf{Model 1}: High-contrast red cubic inclusion with value $10^3$.
(b) \textbf{Model 2}: High-contrast red cubic thin red cylinders with value $10^3$.}
\label{fig:highcon}
\end{figure}

\subsubsection{Convergence test of \textbf{Model~1}}
To further interpret the numerical results, we first examine the electromagnetic field distribution from the perspective of a three-dimensional rods-in-air photonic crystal with randomly distributed high-contrast cubic inclusions, as illustrated in \cref{fig:highcon}-(a). The relative errors in the energy norm and the \( L^2 \) norm for \textbf{Model~1} are reported in \cref{fig:model4-con} and \cref{tab:random1}. In \textbf{Model~1}, we consider a random inclusion configuration, which is commonly adopted in multiscale methods as a benchmark for assessing robustness with respect to nonperiodic and highly heterogeneous coefficient distributions \cite{ye2024multiscale}. As shown in \cref{fig:model4-con}, even for a refined coarse mesh size \( H = 1/16 \), the relative error in the energy norm remains around \(20\%\). {Furthermore, the standard edge-element method on coarse meshes does not exhibit clear convergence behavior, as indicated by the nearly
flat error blue curves. In contrast, the proposed CEM-GMsFEM 
method demonstrates stable convergence and substantially improved accuracy 
when the oversampling size is chosen as $m=4$.} The classical CEM-GMsFEM \cite{chung2018} has been shown to be effective for problems featuring long, high-contrast channels, and the proposed method retains this capability. In general, the relative errors measured in the \( L^2 \) norm are approximately one order of magnitude smaller than those in the energy norm. In particular, for \( m = 3 \) and \( m = 4 \), the proposed method achieves a relative \( L^2 \)-error of approximately \(0.4\%\).

The associated geometric configuration is illustrated by the top view of the photonic crystal structure, where four high-contrast cubic inclusions are embedded at random locations within the background medium. The corresponding two-dimensional field profiles are shown in \cref{fig:placeholder1}. Although the inclusion configuration is not strictly periodic, the presence of strong material contrast still induces pronounced multiple scattering effects. For the considered wave number \( k = 4 \), the interaction between the incident waves and the randomly placed inclusions leads to partial suppression of propagating modes and gives rise to an effective medium behavior at the macroscopic scale. Consequently, the electromagnetic field exhibits a smooth spatial variation in the two-dimensional slices, while fine-scale oscillations associated with individual inclusions are largely averaged out. These results demonstrate that the proposed method remains robust for nonperiodic photonic crystal configurations and successfully reproduces the homogenized electromagnetic response of complex three-dimensional media.
\begin{figure}[!ht]
    \centering
    \includegraphics[width=\linewidth]{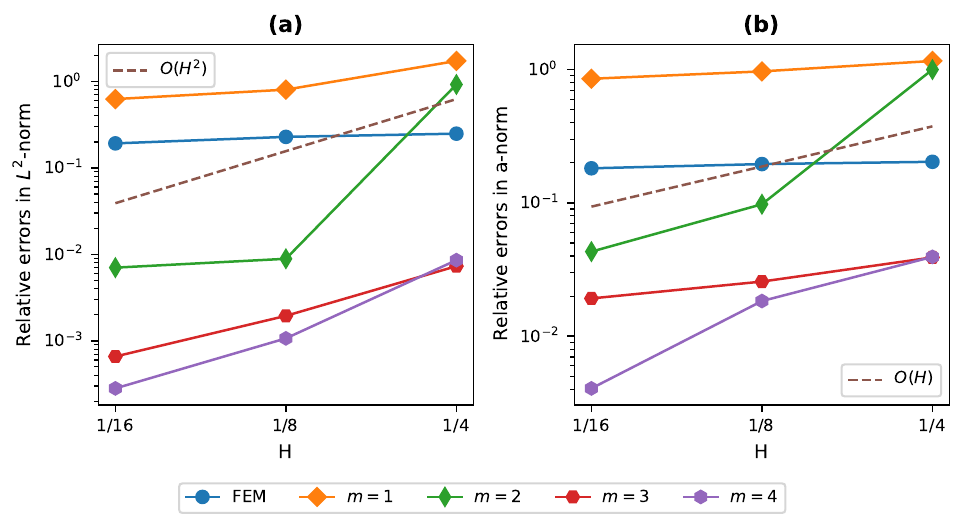}
    \caption{Numerical results for the High-contrast photonic band structures in \textbf{Model 1 
    }. Subplots (a) and (b) show the relative errors of the proposed method with different numbers of oversampling layers $m$ and the FEM  w.r.t. the coarse mesh size $H$, but measured in different norms.}
    \label{fig:model4-con}
\end{figure}
\begin{table}[!ht]
\caption{Relative errors in the a-norm (columns labelled $\|\cdot\|_{a(\Omega)}$) and in the $L^{2}$ norm (columns labelled $\|\cdot\|_{L^{2}(\Omega)}$) for \textbf{Model 1 
}.}
\centering
\resizebox{\textwidth}{!}{
\makegapedcells
\footnotesize
\begin{tabular}{c c c c c c c c c}
\toprule
\multirow{3}{*}{$H$} & \multicolumn{2}{c}{$m=1$} & \multicolumn{2}{c}{$m=2$} & \multicolumn{2}{c}{$m=3$} & \multicolumn{2}{c}{$m=4$} \\
\cmidrule{2-9}
 & $\|\cdot\|_{a(\Omega)}$ & $\|\cdot\|_{L^{2}(\Omega)}$
 & $\|\cdot\|_{a(\Omega)}$ & $\|\cdot\|_{L^{2}(\Omega)}$
 & $\|\cdot\|_{a(\Omega)}$ & $\|\cdot\|_{L^{2}(\Omega)}$
 & $\|\cdot\|_{a(\Omega)}$ & $\|\cdot\|_{L^{2}(\Omega)}$ \\
\midrule
$\frac{1}{4}$ & \num{1.159e+0} & \num{1.725e+0} & \num{9.997e-01} & \num{9.223e-01}&\num{3.902e-02}  &  \num{7.304e-03} & \num{3.939e-02} & \num{8.569e-03} \\
$\frac{1}{8}$  & \num{9.688e-01}& \num{8.038e-01} &\num{ 9.772e-02} & \num{8.888e-03} & \num{2.571e-02} & \num{1.944e-03} & \num{1.837e-02} & \num{1.067e-03} \\
$\frac{1}{16}$  & \num{8.531e-01} & \num{6.253e-01} & \num{4.311e-02} & \num{7.022e-03} & \num{1.922e-02} & \num{6.587e-04} & \num{4.065e-03} & \num{2.816e-04} \\
\bottomrule
\end{tabular}
}
\label{tab:random1}
\end{table}
\begin{table}[!ht]
\centering
\caption{{Comparison of DOFs for the reference solution and the proposed CEM-GMsFEM method with $l_i=4$.}}
\label{DOFs1}
\renewcommand{\arraystretch}{1.2}
\begin{tabular}{lcccc}
\toprule
\multirow{2}{*}{$m=3$} 
& Reference solution 
& \multicolumn{3}{c}{CEM-GMsFEM solution} \\
\cmidrule(lr){3-5}
& $h = 1/64$ 
& $H=1/32$ 
& $H=1/16$ 
& $H=1/4$ \\
\midrule
DOFs 
& 811200 
& 131072
& 16384
& 256 \\
\bottomrule
\end{tabular}
\end{table}
\begin{figure}
    \centering
    \includegraphics[width=\linewidth]{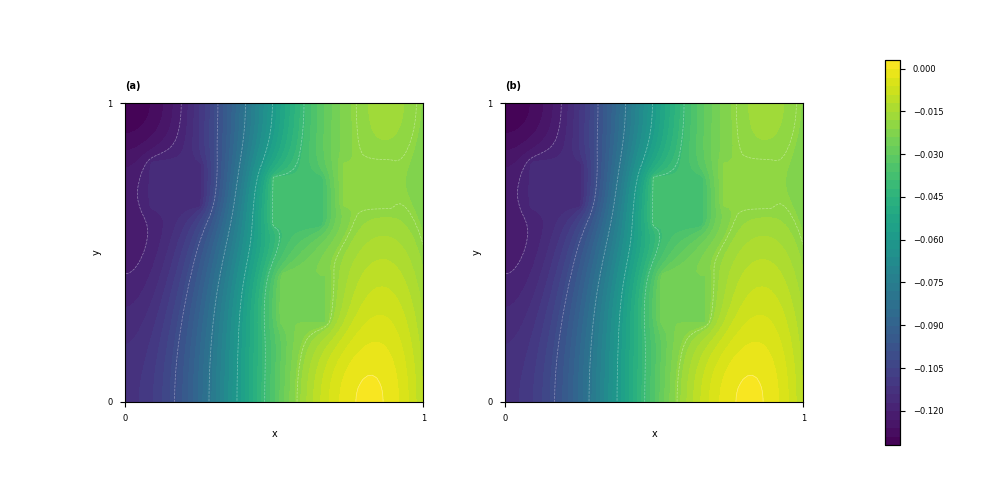}
\caption{Top view of the 3D photonic crystal (rods-in-air) in \textbf{Model 1}. (a) Reference solution obtained by FEM; (b) Solution computed with the proposed multiscale method.}
    \label{fig:placeholder1}
\end{figure}
{\subsubsection{Influence of the number of local basis functions in \textbf{Model~1} }}
{To investigate the influence of the number of local basis functions, we vary 
$l_i \in \{1,2,3,4,5\}$ in \textbf{Model~1} while keeping all other parameters fixed. 
Table~\ref{tab:li_effect} and Table~\ref{DOFs2} report the corresponding relative $L^2$-norm and $a$-norm errors, together with the offline basis construction time, the CEM online solve time, the fine-scale FEM solve time, and degrees of freedom (DOFs). As shown in Figure~\ref{fig:basis}, the error decreases rapidly as $l_i$ increases, 
demonstrating the spectral convergence property of the local approximation space. 
In particular, increasing $l_i$ from $1$ to $2$ leads to a dramatic reduction 
in both the $L^2$-norm error (from $0.2727$ to $0.0296$) and the energy error 
(from $0.5390$ to $0.1069$). A further increase from $l_i=2$ to $l_i=3$ still 
provides noticeable improvement, especially in the energy norm. However, once the dominant local eigenmodes are included, the error decay begins 
to saturate. The improvement from $l_i=3$ to $l_i=4$ is relatively modest, and 
although $l_i=5$ yields additional accuracy, the relative gain is much smaller 
compared with the initial enrichment steps. This behavior confirms that the 
essential multiscale features are captured by only a few carefully selected 
local basis functions.}

{From a computational perspective, increasing $l_i$ enlarges the dimension of 
the multiscale space (see Table~\ref{tab:li_effect}), which leads to higher offline cost. 
This is because larger local eigenvalue problems must be solved and more basis 
functions are constructed over oversampling regions. The offline time increases 
from $17.25$ seconds at $l_i=1$ to $40.65$ seconds at $l_i=5$, reflecting the 
growing complexity of the local spectral problems. In contrast, the online CEM solve time remains very small compared with the 
fine-scale FEM solve time. Even for $l_i=5$, the online time is only $0.0141$ 
seconds, while the fine-scale FEM system (with $811200$ degrees of freedom) 
requires $0.0357$ seconds. The CEM-GMsFEM system, even at its largest size 
($3060$ degrees of freedom), is still two to three orders of magnitude smaller 
than the fine-grid system. This clearly demonstrates the massive dimension 
reduction achieved by the proposed method.}

{Therefore, a clear trade-off exists between accuracy and computational cost. Small values of $l_i$ yield low offline cost but insufficient accuracy, 
whereas large values of $l_i$ significantly increase the offline computational 
burden without proportional improvement in accuracy. Based on these observations, we choose $l_i = 4$ in the subsequent experiments. 
At this level, the error decay is essentially saturated, while both the offline and online
time consuming remains acceptable. Moreover, selecting a slightly larger 
multiscale space enhances stability and robustness, particularly for 
heterogeneous and high-contrast coefficients. Hence, $l_i = 4$ represents a 
balanced choice between accuracy, robustness, and computational efficiency.}
\begin{figure}[!ht]
    \centering
    \includegraphics[width=\linewidth]{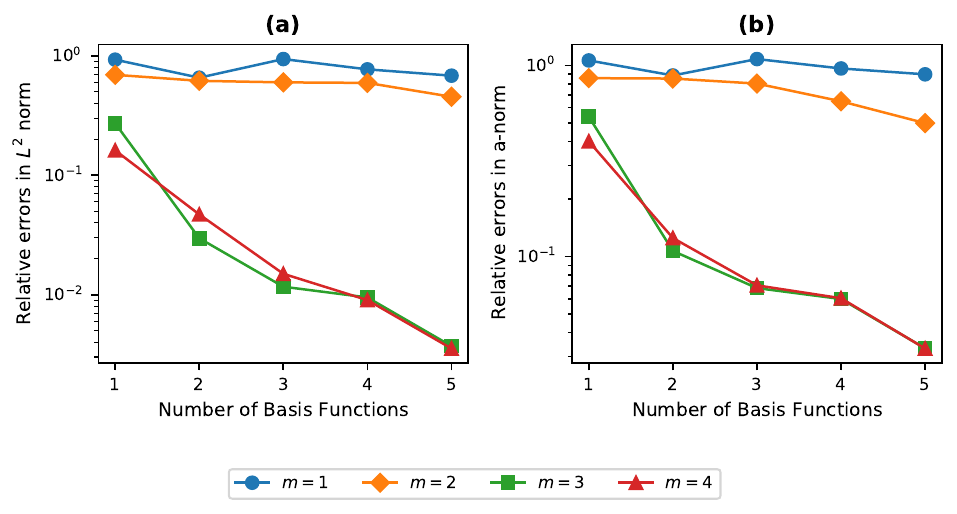}
\caption{The relative errors of the proposed method with different numbers of basis functions $l_i$.}
    \label{fig:basis}
\end{figure}
\begin{table}[!ht]
\centering
\caption{{Influence of the number of local basis functions $l_i$ on accuracy and computational cost ($m=3$).}}
\label{tab:li_effect}
\begin{tabular}{c c c c c c}
\hline
$l_i$ & $L^2$ Error & Energy Error & Offline (s) & Online (s) & FEM (s) \\
\hline
1 & 0.2727 & 0.5390 & 17.2539  & 0.0008 & 0.0357\\
2 & 0.0296 & 0.1069 & 20.0772& 0.0016 &  0.0357 \\
3 & 0.0116 & 0.0683 & 37.8429 & 0.0050 &  0.0357 \\
4 & 0.0094 & 0.0599& 38.8747 & 0.0078& 0.0357\\
5 &  0.0037 & 0.0330 & 40.6454 & 0.0141& 0.0357\\
\hline
\end{tabular}
\end{table}
\begin{table}[!ht]
\centering
\caption{{Comparison of DOFs for the reference solution and the proposed CEM-GMsFEM methods with $H = 1/8$.}}
\label{DOFs2}
\renewcommand{\arraystretch}{1.2}
\begin{tabular}{lcccccc}
\toprule
\multirow{2}{*}{$m=3$} 
& Reference solution 
& \multicolumn{5}{c}{CEM-GMsFEM solution} \\
\cmidrule(lr){3-7}
& $h = 1/64$ 
& $l_i = 1$ 
& $l_i = 2$ 
& $l_i = 3$ 
& $l_i = 4$ 
& $l_i = 5$ \\
\midrule
DOFs 
& 811200 
& 512 
& 1024 
& 1536 
& 2048 
& 3060 \\
\bottomrule
\end{tabular}
\end{table}
\subsubsection{Convergence test of \textbf{Model~2}}
In \textbf{Model~2}, we revisit the periodic cylindrical inclusion model. The coefficient profile \( \mu_r \) is defined as shown in \cref{fig:highcon}-(b), corresponding to a \(5 \times 5\) periodic configuration. We first examine the convergence behavior of the numerical errors in order to verify the theoretical predictions of Theorem~\ref{local convergence}. The corresponding results are presented in \cref{tab:model2} and \cref{fig:model2}. In all tests, we fix \( l_i = 4 \) and vary the number of oversampling layers \( m \) from 1 to 4. As illustrated in \cref{fig:model2}, when \( m = 3 \) and the coarse mesh size \( H \) is refined from \( 1/8 \) to \( 1/16 \), the relative errors in both the energy norm and the \( L^2 \) norm increase. This behavior can be attributed to the local multiscale error term, which contains a factor of \( H^{-1} \). By increasing the number of oversampling layers, the accuracy improves significantly. In particular, for \( H = 1/16 \), the numerical solution achieves a relative error of approximately \(10\%\) in the energy norm, exhibiting first-order convergence in the energy norm and second-order convergence in the \( L^2 \) norm. {In contrast, the standard edge-element method does not demonstrate clear convergence behavior, as evidenced by the nearly flat blue error curves in \cref{fig:model2}}. These observations indicate that the accuracy of the proposed method is jointly governed by the coarse mesh size \( H \) and the number of oversampling layers \( m \), which is fully consistent with the theoretical results established in Theorem~\ref{local convergence}. 

To facilitate a detailed comparison of the solutions and to better resolve the wave propagation behavior, we examine two-dimensional slices of the electromagnetic field on the plane \( z = 0.5 \); see \cref{fig:model22d}. The simulations are performed with fixed parameters \( H = 1/16 \), \( m = 4 \), and wave number \( k = 4 \), corresponding to a frequency range within the photonic band structure. The numerical results exhibit sharp transitions between the matrix and the inclusion regions, reflecting the high contrast in the material coefficients. When the periodic lattice is deliberately perturbed by introducing a defect, a localized resonant mode is formed in which electromagnetic waves with frequencies lying in the photonic band gap are confined and cannot propagate into the surrounding periodic medium. This field localization is clearly observed in the numerical profile: the pronounced peak (yellow) at the defect center indicates strong energy confinement, while the surrounding periodic structure acts as an effective reflective barrier that suppresses wave propagation. This mechanism underlies the fundamental operating principle of photonic crystal cavities.

{In \cref{tab:high-error}, we further investigate the dependence of the numerical errors on the contrast ratio and the number of oversampling layers. As shown in \cref{local convergence} and the subsequent remark, for a given contrast value, a sufficiently large oversampling size is required in order to achieve the 
desired convergence rate. This theoretical prediction is clearly confirmed by 
the numerical results presented in \cref{tab:high-error}. For a fixed high-contrast ratio $\Upsilon$, the accuracy improves as the oversampling size increases. This behavior can be observed in each row of 
\cref{tab:high-error}: when $m$ increases from $2$ to $4$ while keeping 
$\Upsilon$ fixed, both the $L^2$-error and the energy-norm error decrease 
accordingly. On the other hand, for a fixed oversampling size $m$, the performance of the 
scheme deteriorates as the contrast of the medium increases. This trend is 
visible in each column of \cref{tab:high-error}: when $m$ is fixed and 
$\Upsilon$ increases from $10$ to $10^4$, the errors grow. This indicates that 
higher contrast requires larger oversampling regions in order to maintain the 
same level of accuracy. These observations are fully consistent with the theoretical analysis, and 
similar phenomena have also been reported in~\cite{chung2018}.}

\begin{figure}[!ht]
    \centering
    \includegraphics[width=\linewidth]{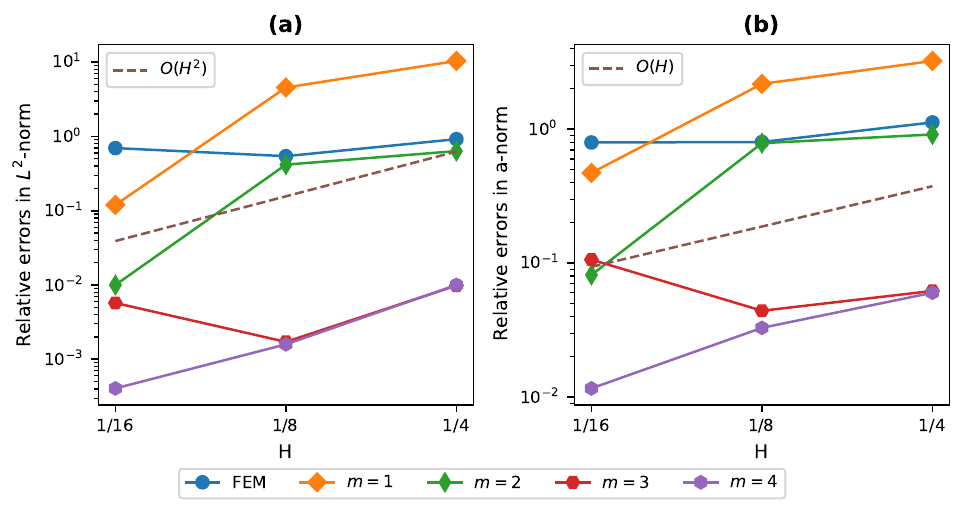}
    \caption{Numerical results for the High-contrast photonic band structures in \textbf{Model 2}. Subplots (a) and (b) show the relative errors of the proposed method with different numbers of oversampling layers $m$ and the FEM w.r.t. the coarse mesh size $H$.}
    \label{fig:model2}
\end{figure}
\begin{table}[!ht]
\caption{Relative errors in the a-norm (columns labelled $\|\cdot\|_{a(\Omega)}$) and in the $L^{2}$ norm (columns labelled $\|\cdot\|_{L^{2}(\Omega)}$) for \textbf{Model 2}.}
\centering
\resizebox{\textwidth}{!}{
\makegapedcells
\footnotesize
\begin{tabular}{c c c c c c c c c}
\toprule
\multirow{3}{*}{$H$} & \multicolumn{2}{c}{$m=1$} & \multicolumn{2}{c}{$m=2$} & \multicolumn{2}{c}{$m=3$} & \multicolumn{2}{c}{$m=4$} \\
\cmidrule{2-9}
 & $\|\cdot\|_{a(\Omega)}$ & $\|\cdot\|_{L^{2}(\Omega)}$
 & $\|\cdot\|_{a(\Omega)}$ & $\|\cdot\|_{L^{2}(\Omega)}$
 & $\|\cdot\|_{a(\Omega)}$ & $\|\cdot\|_{L^{2}(\Omega)}$
 & $\|\cdot\|_{a(\Omega)}$ & $\|\cdot\|_{L^{2}(\Omega)}$ \\
\midrule
$\frac{1}{4}$  & \num{3.230e+0} & 1.036e+1& \num{9.134e-01} & \num{6.347e-01} & \num{6.196e-02} & \num{7.832e-03} & \num{6.007e-02} & \num{1.000e-02} \\
$\frac{1}{8}$  & \num{2.178e+0}& \num{4.555e+0}& \num{7.861e-01 }& \num{4.168e-01} & \num{4.407e-02} & \num{1.715e-03} & \num{3.285e-02} & \num{1.590e-03} \\
$\frac{1}{16}$ & \num{4.715e-01} & \num{1.196e-01} & \num{8.150e-02} & \num{1.000e-02} &\num{1.061e-01}  &  \num{5.734e-03} & \num{1.160e-02} & \num{4.060e-04}\\
\bottomrule
\end{tabular}
}
\label{tab:model2}
\end{table}
\begin{figure}[!ht]
\centering
\includegraphics[width=\linewidth]{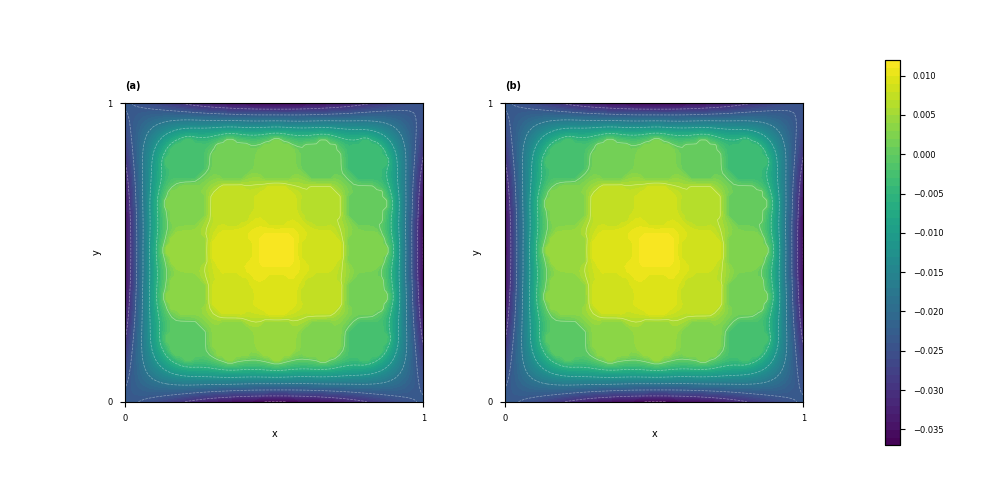}
\caption{Cross-section of the 3D photonic crystal ( holes-in-slab) in \textbf{Model 2}. (a) Reference solution obtained by FEM; (b) Solution computed with the proposed multiscale method.}
\label{fig:model22d}
\end{figure}
\begin{table}[!ht]
\caption{{The relative errors in the $a$-norm (in the columns labelled with $\norm{\cdot}_{a(\Omega)}$) and in the $L^2$ norm (in the columns labelled with $\norm{\cdot}_{L^2(\Omega)}$) with different contrast ratios $\Upsilon$ for \textbf{Model 2}.}}
\centering
\resizebox{\textwidth}{!}{
\makegapedcells
\footnotesize{
\begin{tabular}{ c c c c c c c c c c}
\toprule
\multirow{3}{*}{$\Upsilon$ } & \multirow{3}{*}{$H$}  &
\multirow{3}{*}{$h$} &
\multirow{3}{*}{$k$}

&\multicolumn{2}{c}{$m=2$}    & \multicolumn{2}{c}{$m=3$}  & \multicolumn{2}{c}{$m=4$}\\
\cmidrule{5-10}
& &  &  & $\norm{\cdot}_{a(\Omega)}$ & $\norm{\cdot}_{L^2(\Omega)}$ & $\norm{\cdot}_{a(\Omega)}$ & $\norm{\cdot}_{L^2(\Omega)}$ & $\norm{\cdot}_{a(\Omega)}$ & $\norm{\cdot}_{L^2(\Omega)}$ \\
\midrule            
$10$    &$\frac{1}{8}$&$\frac{1}{64}$ &$4$ & \num{2.363e-1}            & \num{2.842e-2}              & \num{2.940e-2}            & \num{9.020e-4}              & \num{5.315e-3}            & \num{3.620e-4}              \\
$10^2$  &  $\frac{1}{8}$&$\frac{1}{64}$ &$4$              & \num{5.597e-1}            & \num{1.927e-1}              & \num{4.183e-2}            & \num{1.113e-3}              & \num{5.895e-3}            & \num{6.250e-4}              \\
$10^3$   &  $\frac{1}{8}$&$\frac{1}{64}$ &$4$       & \num{7.861e-01}            & \num{4.168e-01}              & \num{4.407e-02}           & \num{1.175e-03}              & \num{3.285e-02}            & \num{1.590e-04}              \\      
$10^4$  &  $\frac{1}{8}$&$\frac{1}{64}$ &$4$                & 0.207e+1           & 0.148e+1             & \num{3.720e-1}            & \num{7.824e-2}              & \num{8.972e-2}            & \num{8.139e-4}              \\
\bottomrule
\end{tabular}
}
}
\label{tab:high-error}
\end{table}

\subsubsection{{Influence of wave number $k$ in \textbf{Model~2}}}
{In this subsection, we investigate the performance of the proposed CEM-GMsFEM for \textbf{Model~2} with more wave numbers \( k \in \{4,8,16,32\} \) in heterogeneous media to further evaluate the robustness of the method. Such regimes are well known to be challenging for standard finite element methods due to the pollution effect and the need for very fine meshes to accurately resolve wave propagation. Meanwhile, we also vary \( l_i \) to examine whether increasing the number of eigenvalues improves the accuracy of the method. 
In the following tests, we fix \( m = 3 \). 
The results are presented in \cref{tab:wave-error-model2}. We assess the robustness of the method with respect to increasing wave numbers. 
In particular, we focus on the decay of the \( L^2 \)-error and the energy-norm error as \( l_i \) increases, as well as the stability of the method for large values of \( k \).  To avoid the pollution effect and to satisfy Assumption~\ref{resolution condition}, sufficiently fine resolutions are required as the wave number increases.
As shown in \cref{tab:wave-error-model2}, both the energy-norm error and the \( L^2 \)-error increase within each column as the wave number grows. 
This behavior is consistent with the theoretical convergence analysis presented in \cref{local convergence}.
 }
\begin{table}[!ht]
\caption{{The relative errors in the $a$-norm (in the columns labelled with $\norm{\cdot}_{a(\Omega)}$) and in the $L^2$ norm (in the columns labelled with $\norm{\cdot}_{L^2(\Omega)}$) with different wave number $k$.}}
\centering
\resizebox{\textwidth}{!}{
\makegapedcells
\footnotesize{
\begin{tabular}{ c c c c c c c c c c}
\toprule
\multirow{3}{*}{$k$} & \multirow{3}{*}{$H$}  & \multicolumn{2}{c}{$l_i=1$}  & \multicolumn{2}{c}{$l_i=2$}    & \multicolumn{2}{c}{$l_i=3$}  & \multicolumn{2}{c}{$l_i=4$}\\
\cmidrule{3-10}
& & $\norm{\cdot}_{a(\Omega)}$ & $\norm{\cdot}_{L^2(\Omega)}$ & $\norm{\cdot}_{a(\Omega)}$ & $\norm{\cdot}_{L^2(\Omega)}$ & $\norm{\cdot}_{a(\Omega)}$ & $\norm{\cdot}_{L^2(\Omega)}$ & $\norm{\cdot}_{a(\Omega)}$ & $\norm{\cdot}_{L^2(\Omega)}$ \\
\midrule            
$4$  &$\frac{1}{4}$  & \num{1.571e-1}   & \num{3.263e-02}              & \num{1.200e-01}            & \num{2.325e-02}              & \num{6.545e-02}            & \num{2.831e-03}              &  \num{6.196e-02} & \num{7.832e-03}         \\
$8$    &$\frac{1}{8}$    & \num{2.461e-1}            & \num{5.840e-2}              & \num{1.771e-1}            & \num{3.007e-2}              & \num{7.781e-2}            & \num{5.855e-3}              & \num{7.527e-2}            & \num{8.407e-3}              \\
$16$   &$\frac{1}{16}$     & \num{3.391e-1}            & \num{1.200e-1}              & \num{2.131e-1}            & \num{7.051e-02}              & \num{8.762e-02}            & \num{3.370e-2}              & \num{7.878e-2}            & \num{8.289e-3}              \\      
$32$    &$\frac{1}{32}$    & \num{4.371e-1}            & \num{1.942e-1}              & \num{2.799e-1}            & \num{8.191e-2}              & \num{9.221e-2}            & \num{4.772e-2}              & \num{8.431e-2}            & \num{9.861e-3}              \\
\bottomrule
\end{tabular}
}
}
\label{tab:wave-error-model2}
\end{table}
{\subsubsection{Spectral problem in \textbf{Model~2} }}
{For the spectral problems \cref{{local spectral problem}}, we consider the coefficients from \textbf{Model~2}, 
which consists of 25 cylindrical inclusions distributed throughout the computational 
domain with a high-contrast ratio $\Upsilon = 10^{3}$. In \cref{tab:spectral}, we present the values of the first four eigenvalues 
$\{\lambda_1, \lambda_2, \lambda_3, \lambda_4\}$ computed from spectral problem over the selected coarse element within $H\in\{1/8, 1/16, 1/32\}$.
From \cref{tab:spectral}, we observe that for $\lambda_1$, the minimum eigenvalues are on the 
order of $10^{-3}$, indicating the presence of very small eigenvalues. 
In contrast, for $\lambda_3$ and $\lambda_4$, the maximum values are significantly 
larger, on the order of $10^{1}$, which is several orders of magnitude greater than 
those of $\lambda_1$. These results clearly demonstrate the existence of a spectral gap structure in the computed eigenvalues, which provides a way to choose the number of auxiliary basis functions adaptively \cite{Xie2026}.}
\begin{table}[!ht]
\caption{
{For \textbf{Model~2} with high-contrast ratio $\Upsilon=10^3$, 
the values of the first four eigenvalues of the marked coarse element.}
}
\label{tab:spectral}
\centering
\makegapedcells
\footnotesize
\begin{tabular}{c c c c c}
\toprule
$H$ & $\lambda_1$ & $\lambda_2$ & $\lambda_3$ & $\lambda_4$ \\
\midrule
$\frac{1}{8}$  & \num{4.266e-02} & \num{4.592e-01} & \num{6.436e-01} & \num{1.627} \\
$\frac{1}{16}$ & \num{1.070e-02} & \num{1.807e-01} & \num{4.434e-01} & \num{10.755} \\
$\frac{1}{32}$ & \num{2.667e-03} & \num{3.308e-01} & \num{3.166e+0}     & \num{3.666} \\
\bottomrule
\end{tabular}
\end{table}
{\subsubsection{Plot of the local eigenfunctions and multiscale basis functions}}
{In order to visualize the local eigenvectors of the previous spectral problem, we continue to use the configuration of \textbf{Model~2}. The top view of the computational domain  with cubic inclusions is shown in 
\cref{fig:maxwell-eigen}\textbf{(a)}. The circled square indicates the local coarse element (the cubic element) on which 
the spectral problem is specifically performed. The values of the first three eigenfunctions computed on the marked coarse element are displayed in \cref{fig:maxwell-eigen}\textbf{(b)/(c)/(d)}. Due to the high-contrast cylindrical inclusions, a blurred circular region can be clearly observed inside the square domain.}

{
In \cref{fig:maxwell-multiscale}, the top view of the corresponding multiscale basis fucntions with different oversampling layers shows that the dominant modes 
are strongly influenced by the high-contrast inclusions. In these regions, 
the multiscale basis functions capture localized multiscale features and indicate that 
the direction of wave propagation changes significantly when passing through 
the high-contrast media. We select the second eigenfunction $\phi_j^{2}$ obtained from the local spectral problem, following the procedure described in \cref{Multiscale basis}, to construct the multiscale basis functions. We consider oversampling layers $m \in \{1,2,3\}$, as illustrated in \cref{fig:maxwell-multiscale}\textbf{(a)--(c)}. Although the resulting multiscale basis functions appear visually similar, the underlying computations are carried out on different oversampled domains. This distinction is highlighted by the red boundaries shown in \cref{fig:maxwell-multiscale}\textbf{(a)--(c)}, which indicate the varying oversampling regions. The location of the selected 
coarse element determines the maximum admissible number of oversampling 
layers, which is $m = 4$.  We treat these multiscale basis functions ($m=4$) as approximations of the global basis functions. We then plot the $L^2$ and energy errors between the multiscale basis functions and the global basis functions.  The results 
demonstrate exponential decay in \cref{fig:maxwell-multiscale} (d), which also confirms \cref{decay}. Based on these observations , our method is able to construct new multiscale basis functions that provide accurate approximations of the global basis functions with reduced model complexity.}
\begin{figure}[!ht]
\centering
 \includegraphics[width=\linewidth]{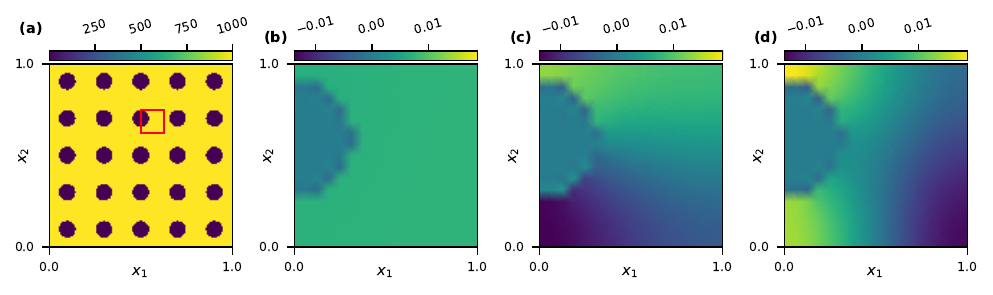}
\caption{   \SubplotTag{(a)} The coefficient profile and the marked coarse element.
    \SubplotTag{(b)}--\SubplotTag{(d)} The plot of the first/second/third eigenfunction corresponding to the marked coarse element. }
 \label{fig:maxwell-eigen}
\end{figure}
\begin{figure}[!ht]
\centering
 \includegraphics[width=\linewidth]{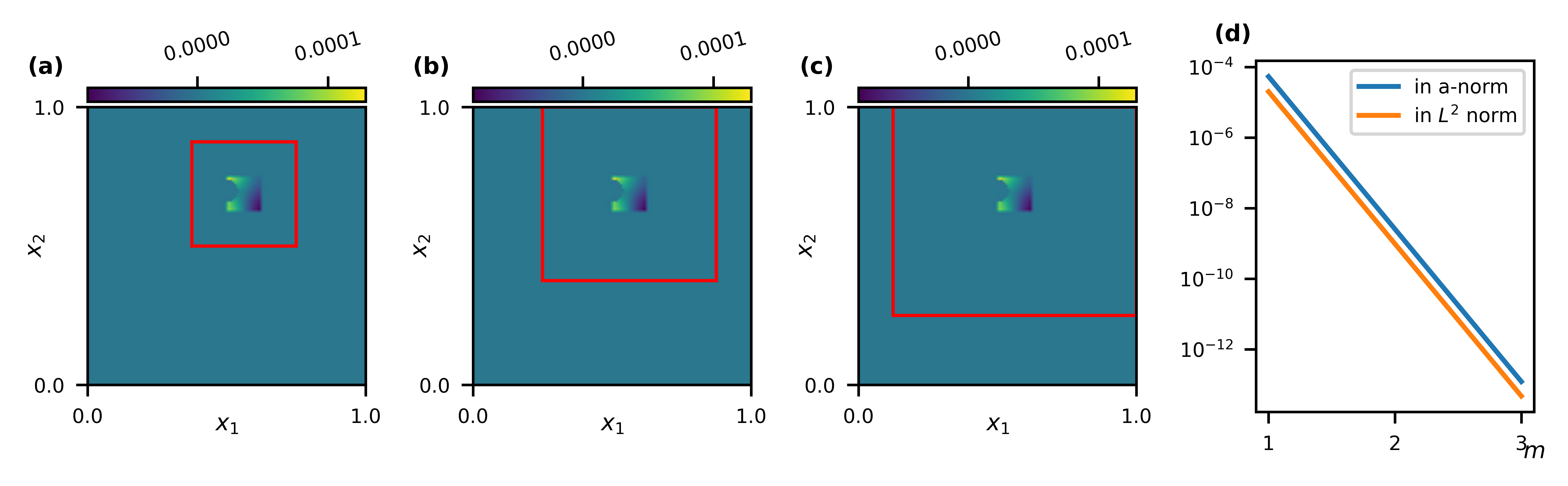}
\caption{   \SubplotTag{(a)} multiscale basis function with $m=1$;   \SubplotTag{(b)} multiscale basis function with $m=2$;  \SubplotTag{(c)} multiscale basis function with $m=3$;   \SubplotTag{(d)} Relative differences
of the multiscale basis functions with $m=1,2,3$ and the global basis functions in the a-norm and $L^2$ norm.
}
 \label{fig:maxwell-multiscale}
\end{figure}

{\subsection{2D high-contrast domain}}
{In this section, we consider the Maxwell problem with a high-contrast coefficient $\mu^{-1}(x,y)$ posed on a two-dimensional suggested in \cite{chung2026multiscalemethodswavepropagation, chung2018} with suitable boundary conditions and source functions. 
The contrast ratio of the medium is set to $\Upsilon = 10^3$ .  To further investigate the performance and scalability of the proposed method, 
we conduct experiments on progressively refined fine grids, with resolutions up to $256^2$ in selected tests. The coarse mesh size $H$ is chosen from the set $\{1/32,\,1/16,\,1/8,\,1/4\}$.}

{In this example, we conduct numerical experiments to assess the robustness of the proposed method shown in Figure~\ref{fig:2d-maxwell}. When $m = 4$, the method achieves relative errors on the order of $10^{-3}$ in both the relative $L^2$ norm and the energy norm for fine coarse meshes, exhibiting clear linear convergence behavior with respect to $H$. 
In \cref{cpu4}, we present detailed CPU time results together with the corresponding numbers of degrees of freedom (DOFs). Based on the linear convergence observed in \cref{fig:2d-maxwell}, we fix $m = 4$. For simplicity, only the online stage is reported, using four basis functions per coarse element. As shown in \cref{cpu4}, CEM-GMsFEM reduces the number of DOFs from $131584$ (FEM) to $4096$ for $H = 1/32$, while maintaining a comparable online computational cost. These results demonstrate that the proposed method achieves high accuracy with a substantially reduced system size. These extended tests confirm that the observed linear convergence in $H$ and exponential decay with respect to $m$ remain stable under further refinement, while the computational savings in DOFs and online CPU time become more pronounced for larger problem sizes. 
These results provide stronger evidence of the efficiency and scalability of the proposed CEM-GMsFEM framework.}
\begin{figure}[!ht]
\centering
 \includegraphics[width=\linewidth]{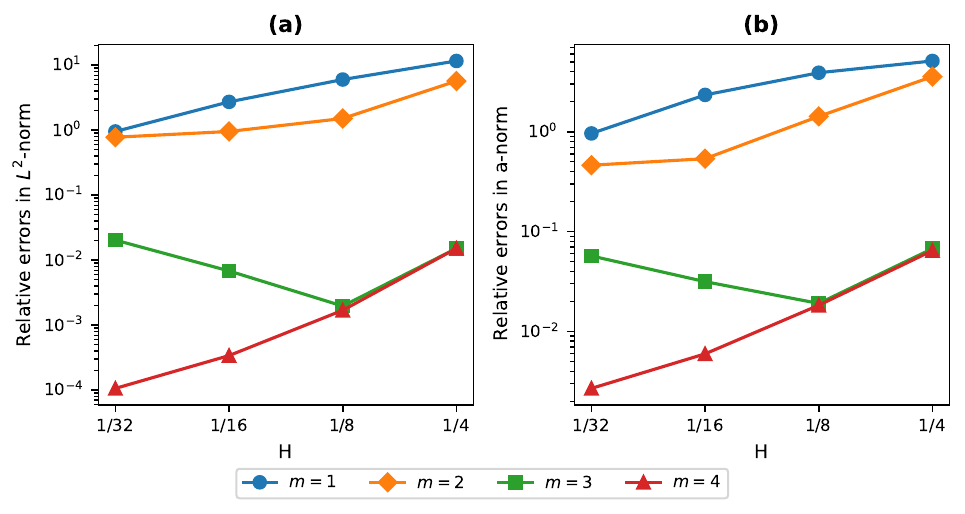}
\caption{ Subplots (a) and (b) show the relative errors of the proposed method with different numbers of oversampling layers $m$ w.r.t. the coarse mesh size $H$, but measured in different norms.}
 \label{fig:2d-maxwell}
\end{figure}
\begin{table}[ht]
\centering
\caption{{Comparison of DOFs and CPU time for the reference solution and the proposed CEM-GMsFEM methods with $l_i=4$.}}
\begin{tabular}{lccccc}
\hline
\multirow{2}{*}{$(m,l_i)= (4,4)$} 
& Reference solution 
& \multicolumn{3}{c}{CEM-GMsFEM solution } \\
\cline{2-6}
& $h=1/256$ 
& $H=1/32$ 
& $H=1/16$ 
& $H=1/8$ 
& $H=1/4$\\
\hline
DOFs 
& 131584
& 4096
&  1024
&  256
&  64 \\  
Time (s) 
&  1.361
& 0.941
&  0.629
&   0.055 
&   0.001\\
\hline
\end{tabular}
\label{cpu4}
\end{table}
\section{Conclusions}
\label{sec:conclusions}
In this paper, we propose a new multiscale method for solving the Maxwell equations. In the construction of the spectral problem, we introduce a global auxiliary space that eliminates the divergence-free constraint by exploiting the fact that the wave number \( k \) is strictly positive. For locally high-contrast media, {the error analysis shows that the proposed multiscale approximation converges to the fine scale solution at a linear rate in the energy norm}. Numerical experiments are presented to validate the theoretical convergence results. The extension of the CEM-GMsFEM framework to higher-order convergence and its application to related wave propagation problems are left for future research.


\section*{Declaration of competing interest}

The authors declare that they have no known competing financial interests or personal relationships that could have appeared
to influence the work reported in this paper.

\section*{Declaration of Generative AI and AI-assisted technologies in the writing process}

During the preparation of this work the authors used ChatGPT in order to improve readability and language. After
using this tool, the authors reviewed and edited the content as needed and take full responsibility for the content of the
publication.

\section*{Acknowledgments}
\input{acknowledgements.tex}

\bibliographystyle{siamplain}
\input{refs-bib-path.tex}

\end{document}

%% file: latex-commands.tex
\usepackage{mathtools}

\allowdisplaybreaks

\usepackage{amssymb}
\usepackage[mathscr]{eucal}

\usepackage{mismath}

\usepackage{hyperref}

\usepackage{bm} 

\usepackage{graphicx}
\graphicspath{{figs/}}

\usepackage{tikz}
\usetikzlibrary{math}
\usetikzlibrary{calc}

\usepackage{subcaption}

\usepackage{multirow}

\usepackage{makecell}
\setcellgapes{2pt}

\usepackage{booktabs}

\usepackage{siunitx}
\usepackage{algorithm}
\usepackage{algpseudocode}

\usepackage[skins]{tcolorbox}
\newtcolorbox{stepbox}[2][]{%
  enhanced,
  attach boxed title to top center={yshift=-3mm,yshifttext=-1mm},
  colframe=blue!75!black,
  colbacktitle=red!80!black,
  fonttitle=\bfseries,
  title=#2,#1
}

\newcommand{\SubplotTag}[1]{\textbf{\small \textsf{#1}}}



%% file: abstract.tex
Modeling time-harmonic Maxwell problems in heterogeneous media presents significant mathematical and computational challenges. Due to the inherent non-elliptic structure and non-coercive nature of Maxwell equations, conventional methods face severe numerical instabilities, particularly in high-contrast media and at high wave numbers. These challenges often lead to ill-conditioned discrete systems and prohibitively high computational costs, limiting their practical applicability.
To overcome these challenges, we introduce an efficient multiscale framework for time-harmonic Maxwell equations with impedance boundary conditions in high-contrast media. A major novelty of this study lies in circumventing the need for an explicit divergence-free constraint on multiscale basis functions. To achieve this, an auxiliary space is constructed via local spectral problems incorporating a mass term and a Silver-Müller-type boundary penalty. This novel design guarantees the coercivity of the corresponding bilinear form and automatically excludes the kernel of the curl operator from the leading eigenspaces. 
Building upon the auxiliary space, we then construct the multiscale space by using a distinct bilinear form. By exploiting a resolution condition and establishing key norm relationships, we rigorously prove the coercivity of this modified bilinear form—a crucial property that underpins the whole analysis. Theoretical analysis shows that, with appropriate oversampling, the method achieves \(O(H)\) convergence independent of the local contrast and the approximation error increases with the wave number $k$. 
Extensive numerical experiments are reported to validate the effectiveness of the proposed approach.

%% file: acknowledgements.tex
Eric T. Chung's work is partially supported by the Hong Kong RGC General Research Fund (Project number: 14304525). Part of this work was completed during Xingguang Jin’s visit to the Hausdorff Research Institute for Mathematics at the University of Bonn. He would like to thank Dr.~Moritz Hauck for his fruitful discussions and the Hausdorff Research Institute for Mathematics for the support provided through the Hausdorff Fellowship.

%% file: refs-bib-path.tex

%% file: main-text.bbl
\begin{thebibliography}{10}

\bibitem{Bensoussan2011}
{A. Bensoussan}, {J.-L. Lions}, and {G. Papanicolaou}, {\em Asymptotic analysis for periodic structures}, AMS Chelsea Publishing, American Mathematical Society, Providence, RI, (2011). Vol.~374.

\bibitem{Cao2010}
{L. Cao}, {Y. Zhang}, {W. Allegretto}, and {Y. Lin}, {\em Multiscale asymptotic method for Maxwell's equations in composite materials}, SIAM J. Numer. Anal. \textbf{47} (2010), pp. 4257--4289.


\bibitem{chung2023multiscale}
{E. Chung, Y. Efendiev and T. Y. Hou}, {\em Multiscale Model Reduction}, Springer, 2023.

\bibitem{chung2018}
{E. T. Chung, Y. Efendiev and W. T. Leung}, {\em Constraint energy minimizing generalized multiscale finite element method}, Computer Methods in Applied Mechanics and Engineering. \textbf{339} (2018), pp. 298--319.

{\bibitem{Chung2019}
{E. T. Chung} and {Y. Li}, {\em Adaptive generalized multiscale finite element methods for H (curl)-elliptic problems with heterogeneous coefficients}, J. Comput. Appl. Math. \textbf{345} (2019), pp. 357--373.}


\bibitem{chung2025locking}
{E. T. Chung, C. Ye and X. Zhong}, {\em A locking free multiscale method for linear elasticity in stress-displacement formulation with high contrast coefficients}, Computer Methods in Applied Mechanics and Engineering. \textbf{447} (2025), pp. 118342.

\bibitem{Chung2026}
{E. T. Chung}, {H. H. Kim} and {X. Zhong}, {\em Iterative contact-resolving hybrid methods for multiscale contact mechanics}, Comput. Methods Appl. Mech. Eng. \textbf{453} (2026), pp. 118843.


\bibitem{Ciarlet2017}
{P. Ciarlet, S. Fliss}, and {C. Stohrer}, {\em On the approximation of electromagnetic fields by edge finite elements. Part 2: A heterogeneous multiscale method for Maxwell's equations}, Comput. Math. Appl. \textbf{73} (2017), pp. 1900--1919.

\bibitem{EW2003}
{W. E and B. Engquist}, {\em The heterognous multiscale methods}, Communications in Mathematical Sciences. \textbf{1(1)} (2003), pp. 87--132.

\bibitem{EWBE2003}
{W. E, B. Engquist and Z. Huang}, {\em Heterogeneous multiscale method: a general methodology for multiscale modeling}, Physical Review B. \textbf{67(9)} (2003), pp. 092101.

\bibitem{YTY2009}
{Y. Efendiev and T. Y. Hou}, {\em Multiscale finite element methods: theory and applications}, Springer Science \& Business Media. \textbf{4} (2009).

\bibitem{BEYH2005}
{B. Engquist and Y. H.~Tsai}, {\em Heterogeneous multiscale methods for stiff ordinary differential equations. Mathematics of computation},  Mathematics of computation. \textbf{74(252)} (2005), pp. 1707--1742.

\bibitem{Galvis2010}
{J. Galvis} and {Y. Efendiev}, {\em  Domain decomposition preconditioners for multiscale flows in high-contrast media}, SIAM Multiscale Modeling \& Simulation. \textbf{8(4)} (2010), pp. 1461--1483.

\bibitem{Henning2016}
{P. Henning}, {M. Ohlberger}, and {B. Verfürth}, {\em A new heterogeneous multiscale method for time-harmonic Maxwell's equations}, SIAM J. Numer. Anal. \textbf{54} (2016), pp. 3493--3522.

\bibitem{Henning2020}
{P. Henning} and {A. Persson}, {\em Computational homogenization of time-harmonic Maxwell's equations}, SIAM Journal on Scientific Computing. \textbf{42(3)} (2020), pp. B581--B607.

\bibitem{Hiptmair2002}
{R. Hiptmair}, {\em Finite elements in computational electromagnetism}, Acta Numer. \textbf{11} (2002), pp. 237--339.


\bibitem{Hochbruck2017}
{M. Hochbruck} and {C. Stohrer}, {\em Finite element heterogeneous multiscale method for time dependent Maxwell's equations}, in Spectral and High Order Methods for Partial Differential Equations---ICOSAHOM 2016, Lect. Notes Comput. Sci. Eng. \textbf{119} (2017), pp. 269--281.

{\bibitem{Holloway2009} {C. L. Holloway, A. Dienstfrey, E. F. Kuester, J. F. O'Hara, A. K. Azad, and A. J. Taylor}, {\em A discussion on the interpretation and characterization of metafilms/metasurfaces: The two-dimensional equivalent of metamaterials}, Metamaterials. \textbf{3(2)} (2009), pp. 100--112.}

\bibitem{HTY1997}
{T. Y. Hou and X. H. Wu}, {\em A multiscale finite element method for elliptic problems in composite materials and porous media}, Journal of computational physics. \textbf{134(1)} (1997), pp. 169--189.

\bibitem{HFJL1998}
{T. J. Hughes, G. R. Feijóo, L. Mazzei, and J. B. Quincy}, {\em The variational multiscale method—a paradigm for computational mechanics}, Computer methods in applied mechanics and engineering. \textbf{166(1-2)} (1998), pp. 3--24.


\bibitem{Johnson1999}
{S. G. Johnson}, {S. Fan}, {P. R. Villeneuve}, {J. D. Joannopoulos}, and {L. A. Kolodziejski}, {\em Guided modes in photonic crystal slabs}, Phys. Rev. B \textbf{60} (1999), pp. 5751--5758.

\bibitem{Lamacz2016}
{A. Lamacz} and {B. Schweizer}, {\em A negative index meta-material for Maxwell's equations}, SIAM J. Math. Anal. \textbf{48(6)} (2016), pp. 4155--4174.

\bibitem{Leonhardt2006}
{U. Leonhardt}, {\em Optical conformal mapping}, Science \textbf{312(5781)} (2006), pp. 1777--1780.


\bibitem{Lipton2018}
{R. Lipton} and {B. Schweizer}, {\em Effective Maxwell's equations for perfectly conducting split ring resonators}, Arch. Ration. Mech. Anal. \textbf{229(3)} (2018), pp. 1197--1221.

\bibitem{Ma2025}
{C. Ma} and {Y. Zhang}, {\em Multiscale model reduction and two-level Schwarz preconditioner for H(curl) elliptic problems}, preprint, arXiv:2506.07381 (2025).

\bibitem{AMDP2014}
{A. Målqvist and D. Peterseim}, {\em Localization of elliptic multiscale problems}, Mathematics of Computation. \textbf{83(290)} (2014), pp. 2583--2603.

\bibitem{Meade2008}
{R. D. Meade}, {S. G. Johnson}, and {J. N. Winn}, {\em Photonic crystals: Molding the flow of light}, Princeton University Press, Princeton, NJ, (2008).

\bibitem{Monk2003}
{P. Monk}, {\em Finite element methods for Maxwell's equations}, Oxford university press. (2003).

\bibitem{Pendry2000}
{J. B. Pendry}, {\em Negative refraction makes a perfect lens}, Phys. Rev. Lett. \textbf{85(18)} (2000), pp. 3966.

{\bibitem{Pendry1999} {J. B. Pendry, A. J. Holden, D. J. Robbins, and W. J. Stewart}, {\em Magnetism from conductors and enhanced nonlinear phenomena}, IEEE Transactions on Microwave Theory and Techniques. \textbf{47(11)} (1999), pp. 2075--2084.}

\bibitem{Peterseim2017}
{D. Peterseim}, {\em Eliminating the pollution effect in Helmholtz problems by local subscale correction}, Mathematics of Computation. \textbf{86(305)} (2017), pp. 1005--1036.

\bibitem{PDR2016}
{D. Peterseim, and R. Scheichl}, {\em Robust numerical upscaling of elliptic multiscale problems at high contrast}, Computational Methods in Applied Mathematics. \textbf{16(4)} (2016), pp. 579--603.


\bibitem{Sakoda2005}
{K. Sakoda}, {\em Optical properties of photonic crystals}, Springer, Berlin, Heidelberg, (2005).

\bibitem{Smith2004}
{D. R. Smith}, {J. B. Pendry}, and {M. C. Wiltshire}, {\em Metamaterials and negative refractive index}, Science \textbf{305(5685)} (2004), pp. 788--792.

\bibitem{Verfuerth2019}
{B. Verf{\"u}rth}, {\em Heterogeneous multiscale method for the Maxwell equations with high contrast}, ESAIM: Math. Model. Numer. Anal. \textbf{53(1)} (2019), pp. 35--61.


\bibitem{Veselago1967}
{V. G. Veselago}, {\em The electrodynamics of substances with simultaneously negative values of \(\varepsilon\) and \(\mu\)}, Uspekhi Fizicheskikh Nauk. \textbf{92(3)} (1967), pp. 517--526.

{\bibitem{Wang2026}
{Y. Wang}, {W. T. Leung}, and {G. Li}, {\em Numerical homogenization for indefinite time-harmonic Maxwell equations}, arXiv preprint arXiv:2604.22502 (2026).}


\bibitem{ye2024multiscale}
{C. Ye}, {X. Jin}, {P. Ciarlet Jr.}, and {E. T. Chung}, {\em Multiscale modeling for a class of high-contrast heterogeneous sign-changing problems}, preprint, arXiv:2407.17130 (2024).

\bibitem{Zhou2026}
{Y. Zhou}, {X. Zhong}, {C. Ye} and {E. T. Chung}, {\em Efficient Multiscale Methods for Highly Heterogeneous Spatial Network Models}, arXiv preprint arXiv:2605.09280 (2026).

\bibitem{JinLiuZhongChung2025}
X.~Jin, L.~Liu, X.~Zhong, and E.~T.~Chung,
{\em Efficient numerical method for the Schr\"{o}dinger equation with high-contrast potentials},
SIAM Multiscale Modeling \& Simulation.
\textbf{23(4)} (2025), pp.~1581--1606.

{\bibitem{Xie2026}
W.~Xie, E.~T.~Chung, Y.~Yang, and Y.~Huang,
\newblock Adaptive multiscale model reduction for linear elasticity equation in perforated domains,
\newblock arXiv preprint arXiv:2606.06839, 2026.}

\bibitem{chung2026multiscalemethodswavepropagation}
E.~T.~Chung, P.~Ciarlet Jr., X.~Jin, and C.~Ye,
\newblock Multiscale Methods for wave propagation in materials with sign-changing coefficients,
\newblock {\em arXiv preprint arXiv:2511.20103}, 2026.
\end{thebibliography}
